\documentclass[11pt]{article}
\usepackage{fullpage} 

\usepackage{ucs}
\usepackage{amssymb}
\usepackage{amsthm}
\usepackage{amsmath}
\usepackage{latexsym}
\usepackage[cp1251]{inputenc}
\usepackage[english]{babel}
\usepackage{graphicx}
\usepackage{wrapfig}
\usepackage{caption}
\usepackage{subcaption}
\usepackage{txfonts}
\usepackage{lmodern}
\usepackage{mathrsfs}
\usepackage{algorithm}
\usepackage{dsfont}
\usepackage{enumerate}
\usepackage{multicol}
\usepackage{csquotes}
\usepackage{stmaryrd}
\usepackage{tikz}
\usepackage{comment}
\usepackage{enumitem}
\usepackage{longtable}
\usepackage{cite}

\usepackage[pdftex,hypertexnames=false,linktocpage=true]{hyperref}
\hypersetup{colorlinks=true,linkcolor=black,anchorcolor=blue,citecolor=black,filecolor=blue,urlcolor=blue,bookmarksnumbered=true,pdfview=FitB}

\def\FF{\mathbb{F}}

\newcommand{\vc}[1]{\ensuremath{\vcenter{\hbox{#1}}}}

\title{Hypergraph Tur\'an Problems in $\ell_2$-Norm}

\newtheorem{theo}{Theorem}

\newtheorem{lemma}[theo]{Lemma}
\newtheorem{ques}[theo]{Question}
\newtheorem{problem}[theo]{Problem}
\newtheorem{corl}[theo]{Corollary}
\newtheorem{conj}[theo]{Conjecture}
\newtheorem{claim}[theo]{Claim}

\theoremstyle{definition}

\newtheorem{defn}[theo]{Definition}
\newtheorem{exmp}[theo]{Example}
\numberwithin{theo}{section}

\tikzset{unlabeled_vertex/.style={inner sep=1.7pt, outer sep=0pt, circle, fill}} 
\tikzset{labeled_vertex/.style={inner sep=2.2pt, outer sep=0pt, rectangle, fill=yellow, draw=black}} 
\tikzset{edge_color0/.style={color=black,line width=1.2pt,opacity=0.5}} 
\tikzset{edge_color1/.style={color=red,  line width=1.2pt,opacity=1}} 
\tikzset{edge_color2/.style={color=blue, line width=1.2pt,opacity=1}} 
\tikzset{edge_color3/.style={color=green,line width=1.2pt}} 
\tikzset{edge_color4/.style={color=red,  line width=1.2pt,dotted}} 
\tikzset{edge_color5/.style={color=blue, line width=1.2pt,dotted}} 
\tikzset{edge_color6/.style={color=green, line width=1.2pt,dotted}} 
\tikzset{edge_color7/.style={color=orange, line width=1.2pt}} 
\tikzset{edge_color8/.style={color=gray, line width=1.2pt}} 
\tikzset{edge_thin/.style={color=black}} 
\tikzset{edge_hidden/.style={color=black,dotted,opacity=0}} 
\tikzset{vertex_color1/.style={inner sep=1.7pt, outer sep=0pt, draw, circle, fill=red}} 
\tikzset{vertex_color2/.style={inner sep=1.7pt, outer sep=0pt, draw, circle, fill=blue}} 
\tikzset{vertex_color3/.style={inner sep=1.7pt, outer sep=0pt, draw, circle, fill=green}} 
\tikzset{labeled_vertex_color1/.style={inner sep=2.2pt, outer sep=0pt, draw, rectangle, fill=red}} 
\tikzset{labeled_vertex_color2/.style={inner sep=2.2pt, outer sep=0pt, draw, rectangle, fill=blue}} 
\tikzset{labeled_vertex_color3/.style={inner sep=2.2pt, outer sep=0pt, draw, rectangle, fill=green}} 
\def\outercycle#1#2{ \draw \foreach \x in {0,1,...,#2}{(0.5*\x,0) coordinate(x\x)}; 
\path (0,0.3) -- (1,0.3); 
} 

\def\drawhypervertex#1#2{ \pgfmathtruncatemacro{\plusone}{#1+1}  \draw[edge_color2] (x#1)++(0,-0.2-0.2*#2)+(-0.2,0) -- +(0.2,0) +(0,0) node[fill=white,outer sep=0,inner sep=0]{\tiny \plusone};}

\def\drawhyperedge#1#2{\pgfmathtruncatemacro{\plusone}{#1+1} \draw[dotted] (x0)++(0,-0.2-0.2*#1)--++(0.5*#2-0.5,0);
\draw[opacity=0] (x0)++(0,-0.2-0.2*#1-0.3)-- ++(-0.4,0);
}

\tikzset{
vtx/.style={inner sep=1.1pt, outer sep=0pt, circle, fill,draw}, 
vtxl/.style={inner sep=1.1pt, outer sep=0pt, rectangle, fill=yellow,draw=black}, 
hyperedge/.style={fill=pink,opacity=0.5,draw=black}, 
}

\newcommand{\Kfourthree}{\vc{
\begin{tikzpicture}[scale=0.55]
\foreach \i in {1,2,3,4}{
\draw (90*\i-45:1.5) coordinate(\i);
}

\foreach \r in {0-45,90-45,180-45,270-45}{
\begin{scope}[rotate=\r]
\draw[hyperedge] 
(0:1.5) to[out=140,in=275,looseness=1.2] (90:1.5)  to[out=265,in=40,looseness=1.2] (180:1.5)  to[out=35,in=145,looseness=1.2] (0:1.5) 
;
\end{scope}
}
\draw 
(1) node[vtx,label=right:{\tiny }]{}
(2) node[vtx,label=left:{\tiny }]{}
(3) node[vtx,label=left:{\tiny }]{}
(4) node[vtx,label=right:{\tiny }]{}
;
\end{tikzpicture}
}}

\newcommand{\Kfourthreeminus}{\vc{
\begin{tikzpicture}[scale=0.55]
\foreach \i in {1,2,3,4}{
\draw (90*\i-45:1.5) coordinate(\i);
}

\foreach \r in {0-45,90-45,270-45}{
\begin{scope}[rotate=\r]
\draw[hyperedge] 
(0:1.5) to[out=140,in=275,looseness=1.2] (90:1.5)  to[out=265,in=40,looseness=1.2] (180:1.5)  to[out=35,in=145,looseness=1.2] (0:1.5) 
;
\end{scope}
}
\draw 
(1) node[vtx,label=right:{\tiny }]{}
(2) node[vtx,label=left:{\tiny }]{}
(3) node[vtx,label=left:{\tiny }]{}
(4) node[vtx,label=right:{\tiny }]{}
;
\end{tikzpicture}
}}

\begin{document}



\author{%
  J\'ozsef Balogh \footnote{Department of Mathematics, University of Illinois at Urbana-Champaign, Urbana, Illinois 61801, USA. E-mail: \texttt{jobal@illinois.edu}. Research is partially supported by NSF Grant DMS-1764123, NSF RTG grant DMS 1937241, Arnold O. Beckman Research
Award (UIUC Campus Research Board RB 18132), the Langan Scholar Fund (UIUC), and the Simons Fellowship.}
\and Felix Christian Clemen \footnote {Department of Mathematics and Statistics, University of Victoria, Victoria, B.C., Canada. E-mail: \texttt{fclemen@uvic.ca}. Research supported in part by PIMS postdoctoral fellowship. }
 \and Bernard Lidick\'{y} \footnote {Iowa State University, Department of Mathematics, Iowa State University, Ames, IA, E-mail: \newline \texttt{ lidicky@iastate.edu}. Research of this author is partially supported by NSF grant DMS-1855653.}
}

\date{\today}
\maketitle


\abstract{
There are various different notions measuring extremality of hypergraphs. In this survey we compare the recently introduced notion of the codegree squared extremal function with the Tur\'an function, the minimum codegree threshold and the uniform Tur\'an density.

The codegree squared sum $\textup{co}_2(G)$ of a $3$-uniform hypergraph $G$ is defined to be the sum of codegrees squared $d(x,y)^2$ over all pairs of vertices $x,y$. In other words, this is the square of the $\ell_2$-norm of the codegree vector. 
We are interested in how large $\textup{co}_2(G)$ can be if we require $G$ to be $H$-free for some $3$-uniform hypergraph $H$. This maximum value of $\textup{co}_2(G)$ over all $H$-free $n$-vertex $3$-uniform hypergraphs $G$ is called the codegree squared extremal function, which we denote by $\textup{exco}_2(n,H)$.

We systemically study the extremal codegree squared sum of various $3$-uniform hypergraphs using various proof techniques. Some of our proofs rely on the flag algebra method while others use more classical tools such as the stability method. In particular, we (asymptotically) determine the codegree squared extremal numbers of matchings, stars, paths, cycles, and $F_5$, the $5$-vertex hypergraph with edge set $\{123,124,345\}$. 

Additionally, our paper has a survey format, as we state several conjectures and give an overview of Tur\'an densities, minimum codegree thresholds and codegree squared extremal numbers of popular hypergraphs. We intend to update the arXiv version of this paper regularly.}
\section{Introduction}
 Given a $k$-uniform hypergraph (or $k$-graph) $H$, the \emph{Tur\'an function} (or \emph{extremal number}) $\textup{ex}(n,H)$ is the maximum number of edges in an $H$-free $n$-vertex $k$-uniform hypergraph. The \emph{Tur\'an density} of $H$, denote by $\pi(H)$, is the scaled limit
 \[
 \pi(H)=\lim_{n\to \infty} \frac{\textup{ex}(n,H)}{ \binom{n}{k}}.
 \]
 Determining these numbers is a central problem in extremal combinatorics. For graphs ($k=2$), this question is well-explored. The Erd\H{o}s-Stone theorem \cite{ErdosStone,ErdosSimonovits} asymptotically determines the Tur\'an density for graphs with chromatic number at least three. For hypergraphs, determining the Tur\'an density is notoriously difficult; very few exact results are known. For example, the Tur\'an density of the innocent looking \emph{tetrahedron} $K_4^3$, the complete $3$-uniform hypergraph on $4$ vertices, is unknown. 
 
 In order to get a better understanding of these problems, various different kinds of extrema\-li\-ty such as the generalized Tur\'an function or the minimum codegree threshold have been studied. We~\cite{BalCleLidmain} recently introduced a new type of extremality for hypergraphs and solved the tetrahedron problem asymptotically for this notion. Here, we will systematically study extremal problems regarding this function. \\
Let $G$ be an $n$-vertex $k$-uniform hypergraph. For a vertex set $T\subset V(G)$, the \emph{codegree} of $T$, denoted by $d_G(T)$, is the number of edges in $G$ containing $T$. We drop the index if $G$ is clear from the context. 
The \emph{codegree vector} of $G$ is the vector $X \in \mathbb{Z}^{\binom{V(G)}{k-1}}$, where $X(v_1,v_2,\ldots,v_{k-1}) = d_G(v_1,v_2,\ldots,v_{k-1})$ for all $\{v_1,v_2,\ldots,v_{k-1}\} \in \binom{V(G)}{k-1}$.
Finding $\textup{ex}(n,H)$ is equivalent to determining the maximum $\ell_1$-norm of the codegree vector of an $H$-free $n$-vertex $k$-uniform hypergraph. Here, we study maximality with respect to the $\ell_2$-norm of the codegree vector. The \emph{codegree squared sum} $\textup{co}_2(G)$ is the sum of codegrees squared over all $k-1$ sets $T$, i.e.,
\begin{align*}
\textup{co}_2(G)=\sum_{\substack{T \in \binom{[n]}{k-1}} }d_G^2(T).
\end{align*}
In other words, the codegree squared sum is the square of the $\ell_2$-norm of the codegree vector. 
\begin{ques}\textup{\cite{BalCleLidmain}}
\label{questionco2}
Given a $k$-uniform hypergraph $H$, what is the maximum $\ell_2$-norm of the codegree vector of a $k$-uniform $H$-free $n$-vertex hypergraph $G$?
\end{ques}
We follow the notation introduced in \cite{BalCleLidmain}. Let $\mathcal{F}$ be a family of $k$-uniform hypergraphs. Denote by $\textup{exco}_2(n,\mathcal{F})$ the maximum codegree squared sum among all $k$-uniform $n$-vertex $\mathcal{F}$-free hypergraphs, and let the codegree squared density $\sigma(\mathcal{F})$ be its scaled limit, i.e.,
\begin{align}
\label{sigmalimit}
    \textup{exco}_2(n,\mathcal{F})= \smash{\displaystyle\max_{G \text{ is } \mathcal{F} \text{-free}}} \textup{co}_2(G) \quad \quad \text{ and } \quad \quad
   \sigma(\mathcal{F})&=\limsup\limits_{n \rightarrow \infty} \frac{\textup{exco}_2(n,\mathcal{F})}{\binom{n}{k-1}(n-k+1)^2}.
\end{align}
We~\cite{BalCleLidmain} proved general properties of $\sigma$ including the existency of the limit in \eqref{sigmalimit}. In Table~\ref{tab:exactcod2} we present bounds and exact values for the codegree squared density of various hypergraphs.
Table~\ref{tab:pic} provides the definitions with pictures of all hypergraphs included in Table~\ref{tab:exactcod2}. Unless otherwise mentioned, all upper bounds on $\sigma$ in this table were obtained using Razborov's flag algebra machinery \cite{flagsRaz}. It is a standard application of flag algebras to obtain these results; we give a short explanation of it in Section~\ref{flagresults}. Table~\ref{tab:exactcod2} also gives an overview of known results for the Tur\'an density, the minimum codegree threshold and the uniform Tur\'an density. Denote by $\delta(G)$ the minimum $(k-1)$-codegree of a $k$-graph $G$. For a family of $k$-graphs $\mathcal{F}$, the \emph{minimum codegree Tur\'an number} $\textup{ex}_{k-1}(n,\mathcal{F})$ is the maximum $\delta(G)$ over all $\mathcal{F}$-free $k$-graphs $G$ on $n$ vertices. The \emph{minimum codegree threshold}, 
\[
\pi_{k-1}(\mathcal{F})= \lim_{n \to \infty} \frac{\textup{ex}_{k-1}(n,\mathcal{F})}{ \binom{n}{k-2}},
\]
is its scaled limit. \\
Reiher, R\"{o}dl and Schacht~\cite{MR3764068} recently introduced a variant of the Tur\'an density, where we want to maximize the density of every linear sized subsets of $H$-free hypergraphs. For real numbers $d\in[0,1]$, and $\eta>0$ a $3$-graph $G=(V,E)$ is \emph{$(d,\eta,1)$-dense} if for all $U\subseteq V$ the relation
\[
\left| U^{(3)}\cap E \right| \geq d \binom{|U|}{3}-\eta |V|^3
\]
holds, where $U^{(3)}$ denotes the set of all three element subsets of $U$. The \emph{uniform Tur\'an density} $\pi_u(H)$ of a 3-graph $H$ is defined to be 
\begin{align*}
\pi_u(H):&=\sup \{d\in [0,1]: \ \text{ for every } \eta>0 \text{ and } n\in \mathbb{N} \text{ there exists an } \text{$H$-free} \\
& \quad \ \ \text{$(d,\eta,1)$-dense hypergraph } G \text{ with } |V(G)|\geq n\}.
\end{align*}
\begin{theo}
\label{tabletheo}
All bounds presented in Table~\ref{tab:exactcod2} hold. 
\end{theo}
Here we collect many known results, and when known results are lacking, we at least mention the `trivial' bounds. Most upper bounds were obtained by a simple application of flag algebras.

In addition to the bounds presented in Table~\ref{tab:exactcod2}, we asymptotically determine the maximum $\ell_2$-norm of $3$-graphs not containing a loose cycle, loose path, matching or star. These problems are not approachable with flag algebra methods due to the fact that their codegree squared extremal number is $o(n^4)$. We use non-computer assisted methods to obtain these results. The discussion about the history of these problems will be deferred to the corresponding sections. Additionally, we provide a non-computer assisted proof determining the exact codegree squared extremal number of $F_5$.

Denote by $S_n$ the complete $3$-partite $3$-graph on $n$ vertices with part sizes $\lfloor n/3\rfloor,\lfloor (n+1)/3\rfloor,$ $\lfloor (n+2)/3\rfloor$.  
We~\cite{BalCleLidmain} showed that $S_n$ is the largest $\{F_4,F_5\}$-free $3$-graph in $\ell_2$-norm using a simple double counting argument and the corresponding $\ell_1$-norm result by Bollob\'as~\cite{Bollobascancellative}. Here, we will expand this result for $F_5$-free $3$-graphs, which requires more work than just applying the corresponding $\ell_1$-norm result.

\begin{table}[H]
\footnotesize
  \def\arraystretch{1.4}
  \setlength{\tabcolsep}{3pt}  
\begin{center}
  \begin{tabular}{ | l|| c |  c | c  |  c | c | c| c| c| }
    \hline
    $H$ &   $\pi(H)\geq $   &  $\pi(H)\leq $ &  $\sigma(H)\geq $   &  $\sigma(H)\leq $ & $\pi_2(H)\geq$ & $\pi_2(H)\leq$ & $\pi_u(H)\geq$ & $\pi_u(H)\leq$ \\
    \hline\hline
    \hyperref[K43]{$K_4^3$} &$5/9$\cite{MR177847}&$0.5615$\cite{BaberTuran} & $1/3$\cite{BalCleLidmain}  & $1/3$\cite{BalCleLidmain} & $1/2$ \cite{codegreeconj} & 0.529  & $1/2$&  0.529   \\
    \hyperref[K53]{$K_5^3$} &$3/4$\cite{Sidorenko1981SystemsOS,MR177847}&  
    $0.7696$~\cite{flagmatic} & $5/8$\cite{BalCleLidmain}  & $5/8$\cite{BalCleLidmain}  & $2/3$ \cite{LoMark,codegreeFalgas}& $0.74$    & $2/3$& 0.758  
    \\
            \hyperref[K63]{$K_6^3$} &$0.84$\cite{Sidorenko1981SystemsOS,MR177847}& 
            $0.8584$\cite{flagmatic}& $0.7348$\cite{BalCleLidmain} & 
            
            $0.7536$ & $3/4$ \cite{LoMark,codegreeFalgas}& $0.838$ & $3/4$ & 0.853 
            \\
    \hyperref[F32]{$F_{3,2}$} &$4/9$ \cite{F32Mubayi}&$4/9$~\cite{F32furedi}& $1/4$  & $1/4+10^{-9}$& $1/3$~\cite{codF32falgas}&$1/3$ \cite{codF32falgas}    & $0$& $0$~\cite{MR3764068}      \\

    \hyperref[sec:F33]{$F_{3,3}$} &$3/4$\cite{F32Mubayi}&$3/4$\cite{F32Mubayi}& $5/8$\cite{BalCleLidmain}  & $5/8$\cite{BalCleLidmain}& $1/2$& $0.604$         & $1/4$&   $1/4$~\cite{Schulke23}    \\
        \hyperref[F5]{$F_{5}$} &$2/9$ \cite{Bollobascancellative} &$2/9$ \cite{F5Frankl}& $2/27$ & $2/27$& $0$& $0$ & $0$ & $0$~\cite{MR3764068} \\ 
                \hyperref[Fano]{$\mathbb{F}$} &$3/4$ \cite{VeraSos}&$3/4$ \cite{FanoFuredi}& $5/8$ & $3/4$ & $1/2$\cite{MubayiFano} &$1/2$ \cite{MubayiFano}& 0 & 0~\cite{MR3764068}\\
            \hyperref[K43minus]{$K_4^{3-}$} & $2/7$ \cite{K43-extremalfrankl}& 
             $0.28689$~\cite{flagmatic}& $4/43$ & 
             $0.09307$& $1/4$ \cite{codegreeconj}&$1/4$\cite{FalgasK4-} & $1/4$~\cite{MR3474967} & $1/4$~\cite{MR3474967,MR3790065}   \\
     \hyperref[K43minusF32C5]{$K_4^{3-},F_{3,2},C_5$} &$12/49$ \cite{RavryTuran}&$12/49$ \cite{RavryTuran}& $2/27$ & $2/27$ & $1/12$& $0.186$  & $0$ & $0$~\cite{MR3764068} \\
    \hyperref[K43-F32]{$K_4^{3-},F_{3,2}$} &$5/18$ \cite{RavryTuran}&$5/18$ \cite{RavryTuran}& $5/54$ & $5/54$& $1/12$ & 0.202  & $0$ & $0$~\cite{MR3764068} \\
  \hyperref[K43-C5]{$K_{4}^{3-},C_5$} &$1/4$ \cite{RavryTuran}& 
  $0.25108$~\cite{RavryTuran}& $1/13$ & 
  $0.07695$& $1/12$ & $0.204$& $1/27$ & $4/27$~\cite{personalComSamuel} \\

    \hyperref[F32-J4]{$F_{3,2},J_4$} &$3/8$ \cite{RavryTuran}& $3/8$ \cite{RavryTuran}& $3/16$ & $3/16$& $1/12$& $0.274$  & $0$ & $0$~\cite{MR3764068} \\ 
    
 \hyperref[F32-J5]{$F_{3,2},J_5$} &$3/8$ \cite{RavryTuran}& $3/8$ \cite{RavryTuran}& $3/16$ &  $3/16$ & $1/12$ & $0.28$ & $0$ & $0$~\cite{MR3764068} \\
    
        \hyperref[sec:J4]{$J_4$} &$1/2$ \cite{DaisyBollobas}& 
        $0.50409$~\cite{flagmatic}& $0.28$ & 
        $0.2808$& $1/4$& $0.473$ & $1/3$~\cite{MR3790065}& 4/9~\cite{MR3790065} \\

        \hyperref[sec:J5]{$J_5$} &$1/2$ \cite{DaisyBollobas} & 
        $0.64475$& $0.28$ &  
        $0.44275$& $1/4$&  $0.613$ & $7/16$~\cite{MR3790065}& $9/16$~\cite{MR3790065}  \\
        
        \hyperref[sec:C5]{$C_5$} &$2\sqrt{3}-3$ \cite{F32Mubayi}&
        $0.46829$~\cite{flagmatic}& $0.25194$ & 
        $0.25311$& $1/3$&$0.3993$ & $4/27$\cite{MR4111729} & $4/27$~\cite{personalComSamuel} \\
        \hyperref[sec:C5minus]{$C_5^-$} &$1/4$ \cite{F32Mubayi}& 
        $1/4$~\cite{BernardFlo,OlegXizhi} &  $1/13$ &  
        $0.07726$ &$0$&  $0$~\cite{PiSaSc} & $0$ & $0$~\cite{MR3764068} \\
        \hyperref[sec:K_5^=]{$K_5^{=}$} &$0.58656$&
        $0.60962$& 
        $0.35794$& 
        $0.38873$& $1/2$&$0.569$& $1/2$& 0.567 \\
       \hyperref[sec:K_5^<]{$K_5^{<}$} &$5/9$&
       $0.57705$& $1/3$ & 
       $0.34022$& $1/2$&$0.560$& $1/2$& 0.568 \\
        \hyperref[sec:K_5^-]{$K_5^{-}$}         &$0.602673$~\cite{HongBjarne} &
        $0.64209$& 
        $0.35794$& 
        $0.41962$& $1/2$& 0.621 & $1/2$& 0.626 \\ 

\hline
  \end{tabular}
\end{center}
\caption[Known bounds]{\label{tab:exactcod2}
  This table shows upper and lower bounds on $\sigma$ and displays the known results for $\pi, \pi_2$ and $\pi_u$ for various hypergraphs.
  For the definition of the hypergraphs see Table~\ref{tab:3uniformfig}.
  Useful approximations: $2/7\approx 0.285714$,
  $2\sqrt{3}-3\approx 0.4641016$,
  $1/13\approx 0.0769230$,
  $4/43\approx 0.0930232$,
  $1/12\approx 0.0833333$.
  }
\end{table}

\newcommand\T{\rule{0pt}{2.6ex}}       
\newcommand\B{\rule[-1.2ex]{0pt}{0pt}} 

\begin{center}
  \setlength{\tabcolsep}{1pt}
  \renewcommand{\arraystretch}{2.0}
  \begin{longtable}{ | l|| c |  c | c  |   }
    \hline
    $H$ &   Edges   &  Visual 1 &  Visual 2   \\
    \hline\hline
    $K_4^3$ 
    &123, 124, 134, 234 
    &
\vc{\begin{tikzpicture}\outercycle{5}{4}
\drawhyperedge{0}{4}
\drawhypervertex{0}{0}
\drawhypervertex{1}{0}
\drawhypervertex{2}{0}
\drawhyperedge{1}{4}
\drawhypervertex{0}{1}
\drawhypervertex{1}{1}
\drawhypervertex{3}{1}
\drawhyperedge{2}{4}
\drawhypervertex{0}{2}
\drawhypervertex{2}{2}
\drawhypervertex{3}{2}
\drawhyperedge{3}{4}
\drawhypervertex{1}{3}
\drawhypervertex{2}{3}
\drawhypervertex{3}{3}
\end{tikzpicture} 
}    
    &
\vc{
\begin{tikzpicture}[scale=0.55]
\foreach \i in {1,2,3,4}{
\draw (90*\i-45:1.5) coordinate(\i);
}

\foreach \r in {0-45,90-45,180-45,270-45}{
\begin{scope}[rotate=\r]
\draw[hyperedge] 
(0:1.5) to[out=140,in=275,looseness=1.2] (90:1.5)  to[out=265,in=40,looseness=1.2] (180:1.5)  to[out=35,in=145,looseness=1.2] (0:1.5) 
;
\end{scope}
}

\draw 
(1) node[vtx,label=right:{\tiny 1}]{}
(2) node[vtx,label=left:{\tiny 2}]{}
(3) node[vtx,label=left:{\tiny 3}]{}
(4) node[vtx,label=right:{\tiny 4}]{}
;
\end{tikzpicture}
}      
   \\
\hline
    $K_4^{3-} = F_4$ 
    &123, 124, 134
    &
\vc{\begin{tikzpicture}\outercycle{5}{4}
\drawhyperedge{0}{4}
\drawhypervertex{0}{0}
\drawhypervertex{1}{0}
\drawhypervertex{2}{0}
\drawhyperedge{1}{4}
\drawhypervertex{0}{1}
\drawhypervertex{1}{1}
\drawhypervertex{3}{1}
\drawhyperedge{2}{4}
\drawhypervertex{0}{2}
\drawhypervertex{2}{2}
\drawhypervertex{3}{2}
\path (0,0.4) -- (0,-1.1); 
\end{tikzpicture} 
}    
    &
\vc{
\begin{tikzpicture}[scale=0.55]
\foreach \i in {1,2,3,4}{
\draw (90*\i-45:1.5) coordinate(\i);
}

\foreach \r in {0-45,90-45,270-45}{
\begin{scope}[rotate=\r]
\draw[hyperedge] 
(0:1.5) to[out=140,in=275,looseness=1.2] (90:1.5)  to[out=265,in=40,looseness=1.2] (180:1.5)  to[out=35,in=145,looseness=1.2] (0:1.5) 
;
\end{scope}
}

\draw 
(1) node[vtx,label=right:{\tiny 1}]{}
(2) node[vtx,label=left:{\tiny 2}]{}
(3) node[vtx,label=left:{\tiny 3}]{}
(4) node[vtx,label=right:{\tiny 4}]{}
;
\end{tikzpicture}
}      
   \\
   \hline
$F_{3,2}$ & 123, 145, 245, 345
&
\vc{\begin{tikzpicture}\outercycle{6}{5}
\drawhyperedge{0}{5}
\drawhypervertex{0}{0}
\drawhypervertex{1}{0}
\drawhypervertex{2}{0}
\drawhyperedge{1}{5}
\drawhypervertex{0}{1}
\drawhypervertex{3}{1}
\drawhypervertex{4}{1}
\drawhyperedge{2}{5}
\drawhypervertex{1}{2}
\drawhypervertex{3}{2}
\drawhypervertex{4}{2}
\drawhyperedge{3}{5}
\drawhypervertex{2}{3}
\drawhypervertex{3}{3}
\drawhypervertex{4}{3}
\end{tikzpicture} 
}
&
\vc{
\vc{
\begin{tikzpicture}
\draw
(0,-0.7) coordinate(1) node[vtx,label=left:{\tiny $1$}](a){}
(0,0) coordinate(2) node[vtx,label=below:{\tiny $2$}](b){}
(0,0.7) coordinate(3) node[vtx,label=left:{\tiny $3$}](c){}
(1,-0.5) coordinate(4) node[vtx,label=right:{\tiny $5$}](d){}
(1,0.5) coordinate(5) node[vtx,label=right:{\tiny $4$}](d){}
;
\draw[hyperedge] (1) to[bend left] (2) to[bend left] (3) to[bend right] (1);
\draw[hyperedge] (1) to[out=0,in=190] (4) to[out=190,in=250] (5) to[out=250,in=0] (1);
\draw[hyperedge] (2) to[out=0,in=140] (4) to[out=140,in=210,looseness=1.5] (5) to[out=210,in=0] (2);
\draw[hyperedge] (3) to[out=0,in=120,looseness=1.2] (4) to[out=120,in=170,looseness=1.2] (5) to[out=170,in=0,looseness=1.2] (3);
\end{tikzpicture}
}
}
   \\
   \hline
   $J_4$
   &
   123, 124, 125, 134, 135, 145
   &
\vc{\begin{tikzpicture}\outercycle{6}{5}
\drawhyperedge{0}{5}
\drawhypervertex{0}{0}
\drawhypervertex{1}{0}
\drawhypervertex{2}{0}
\drawhyperedge{1}{5}
\drawhypervertex{0}{1}
\drawhypervertex{1}{1}
\drawhypervertex{3}{1}
\drawhyperedge{2}{5}
\drawhypervertex{0}{2}
\drawhypervertex{1}{2}
\drawhypervertex{4}{2}
\drawhyperedge{3}{5}
\drawhypervertex{0}{3}
\drawhypervertex{2}{3}
\drawhypervertex{3}{3}
\drawhyperedge{4}{5}
\drawhypervertex{0}{4}
\drawhypervertex{2}{4}
\drawhypervertex{4}{4}
\drawhyperedge{5}{5}
\drawhypervertex{0}{5}
\drawhypervertex{3}{5}
\drawhypervertex{4}{5}
\end{tikzpicture} 
}   
   &
\vc{
\begin{tikzpicture}[scale=0.75]
\clip (-2.5,-1.1) rectangle (1.5,1.1);
\draw
(-2,0) coordinate(1) node[vtx,label=left:{\tiny $1$}](b){}
(45:1) coordinate(2) node[vtx,label=right:{\tiny $2$}](c){}
(135:1) coordinate(3) node[vtx,label=left:{\tiny $3$}](d){}
(225:1) coordinate(4) node[vtx,label=left:{\tiny $4$}](d){}
(315:1) coordinate(5) node[vtx,label=right:{\tiny $5$}](d){}
;
\draw[hyperedge] (1) to[out=70,in=140,looseness=0.9] (2) to[out=140,in=60,looseness=0.8]
(3) to[out=120,in=70,looseness=0.8] (1);
\draw[hyperedge] (1) to[out=0,in=250,looseness=1.3] (3) to[out=250,in=110] 
(4) to[out=110,in=0,looseness=1.3] (1);
\draw[hyperedge] (1) to[out=-70,in=240,looseness=0.8] (4) to[out=300,in=220,looseness=0.8] 
(5) to[out=220,in=-70,looseness=0.9] (1);
\draw[hyperedge] (1) to[out=0,in=110,looseness=0.7] (5) to[out=110,in=250,looseness=1.1] 
(2) to[out=250,in=0,looseness=0.7] (1);
\draw[hyperedge] (1) to[out=-25,in=205,looseness=0.8] (2) to[out=205,in=110,looseness=0.8] 
(4) to[out=110,in=-25,looseness=0.9] (1);
\draw[hyperedge] (1) to[out=25,in=250,looseness=0.9] (3) to[out=250,in=155,looseness=0.8] 
(5) to[out=155,in=25,looseness=0.8] (1);
\draw[line width = 0.5pt]
(3) to[out=250,in=155,looseness=0.8] (5)
(2) to[out=205,in=110,looseness=0.8] (4)
(5) to[out=110,in=250,looseness=1.1] (2)
(4) to[out=300,in=220,looseness=0.8] (5)
(3) to[out=250,in=110] (4)
(2) to[out=140,in=60,looseness=0.8] (3)
;
\end{tikzpicture}
}   
   \\
   \hline

   $J_5$
   &
   123, 124, 125, 126, 134, 135, 136, 145, 146, 156
   &
\vc{\begin{tikzpicture}\outercycle{7}{6}
\drawhyperedge{0}{6}
\drawhypervertex{0}{0}
\drawhypervertex{1}{0}
\drawhypervertex{2}{0}
\drawhyperedge{1}{6}
\drawhypervertex{0}{1}
\drawhypervertex{1}{1}
\drawhypervertex{3}{1}
\drawhyperedge{2}{6}
\drawhypervertex{0}{2}
\drawhypervertex{1}{2}
\drawhypervertex{4}{2}
\drawhyperedge{3}{6}
\drawhypervertex{0}{3}
\drawhypervertex{1}{3}
\drawhypervertex{5}{3}
\drawhyperedge{4}{6}
\drawhypervertex{0}{4}
\drawhypervertex{2}{4}
\drawhypervertex{3}{4}
\drawhyperedge{5}{6}
\drawhypervertex{0}{5}
\drawhypervertex{2}{5}
\drawhypervertex{4}{5}
\drawhyperedge{6}{6}
\drawhypervertex{0}{6}
\drawhypervertex{2}{6}
\drawhypervertex{5}{6}
\drawhyperedge{7}{6}
\drawhypervertex{0}{7}
\drawhypervertex{3}{7}
\drawhypervertex{4}{7}
\drawhyperedge{8}{6}
\drawhypervertex{0}{8}
\drawhypervertex{3}{8}
\drawhypervertex{5}{8}
\drawhyperedge{9}{6}
\drawhypervertex{0}{9}
\drawhypervertex{4}{9}
\drawhypervertex{5}{9}
\end{tikzpicture} 
}
&
\vc{
\begin{tikzpicture}
\draw
(-2,0) coordinate(1) node[vtx,label=left:{\tiny $1$}](1){}
\foreach \x in {2,...,6}
{
(-144+72*\x:1) coordinate(\x) node[vtx](v\x){}
}
;
\draw[hyperedge,opacity=0.2] (1)--(2)--(3)--(1);
\draw[hyperedge,opacity=0.2] (1)--(2)--(4)--(1);
\draw[hyperedge,opacity=0.2] (1)--(2)--(5)--(1);
\draw[hyperedge,opacity=0.2] (1)--(2)--(6)--(1);
\draw[hyperedge,opacity=0.2] (1)--(3)--(4)--(1);
\draw[hyperedge,opacity=0.2] (1)--(3)--(5)--(1);
\draw[hyperedge,opacity=0.2] (1)--(3)--(6)--(1);
\draw[hyperedge,opacity=0.2] (1)--(4)--(5)--(1);
\draw[hyperedge,opacity=0.2] (1)--(4)--(6)--(1);
\draw[hyperedge,opacity=0.2] (1)--(5)--(6)--(1);
\draw [line width = 0.5pt] (2)--(3)--(4)--(5)--(6)--(2)--(4)--(6)--(3)--(5)--(2)  ;
\draw (-2,0) coordinate(1) node[vtx](1){}
\foreach \x in {2,...,6}
{
(-144+72*\x:1) coordinate(\x) node[vtx](v\x){}
};
\end{tikzpicture}
}
   \\
   \hline

$F_{3,3}$
&
 123, 145, 146, 156, 245, 246, 256, 345, 346, 356
 &
\vc{\begin{tikzpicture}\outercycle{7}{6}
\drawhyperedge{0}{6}
\drawhypervertex{0}{0}
\drawhypervertex{1}{0}
\drawhypervertex{2}{0}
\drawhyperedge{1}{6}
\drawhypervertex{0}{1}
\drawhypervertex{3}{1}
\drawhypervertex{4}{1}
\drawhyperedge{2}{6}
\drawhypervertex{0}{2}
\drawhypervertex{3}{2}
\drawhypervertex{5}{2}
\drawhyperedge{3}{6}
\drawhypervertex{0}{3}
\drawhypervertex{4}{3}
\drawhypervertex{5}{3}
\drawhyperedge{4}{6}
\drawhypervertex{1}{4}
\drawhypervertex{3}{4}
\drawhypervertex{4}{4}
\drawhyperedge{5}{6}
\drawhypervertex{1}{5}
\drawhypervertex{3}{5}
\drawhypervertex{5}{5}
\drawhyperedge{6}{6}
\drawhypervertex{1}{6}
\drawhypervertex{4}{6}
\drawhypervertex{5}{6}
\drawhyperedge{7}{6}
\drawhypervertex{2}{7}
\drawhypervertex{3}{7}
\drawhypervertex{4}{7}
\drawhyperedge{8}{6}
\drawhypervertex{2}{8}
\drawhypervertex{3}{8}
\drawhypervertex{5}{8}
\drawhyperedge{9}{6}
\drawhypervertex{2}{9}
\drawhypervertex{4}{9}
\drawhypervertex{5}{9}
\end{tikzpicture} 
}

 &
\vc{
\begin{tikzpicture}
\draw
(0,-0.7) coordinate(1) node[vtx,label=left:{\tiny $1$}](a){}
(0,0) coordinate(2) node[vtx,label=left:{\tiny $2$}](b){}
(0,0.7) coordinate(3) node[vtx,label=left:{\tiny $3$}](c){}
(1,-0.7) coordinate(4) node[vtx,label=right:{\tiny $4$}](d){}
(1,0.0) coordinate(5) node[vtx,label=right:{\tiny $5$}](d){}
(1,0.7) coordinate(6) node[vtx,label=right:{\tiny $6$}](d){}
;
\draw[hyperedge] (1) to[bend left] (2) to[bend left] (3) to[bend right] (1);
\draw[hyperedge] (1) to[out=0,in=190] (4) to[out=190,in=250] (5) to[out=250,in=0] (1);
\draw[hyperedge] (2) to[out=0,in=140] (4) to[out=140,in=210,looseness=1.5] (5) to[out=210,in=0] (2);
\draw[hyperedge] (3) to[out=-60,in=120,looseness=1.2] (4) to[out=120,in=170,looseness=1.2] (5) to[out=170,in=-60,looseness=1.2] (3);
\draw[hyperedge] (1) to[out=60,in=190] (5) to[out=190,in=250] (6) to[out=250,in=60] (1);
\draw[hyperedge] (2) to[out=0,in=140] (5) to[out=140,in=210,looseness=1.5] (6) to[out=210,in=0] (2);
\draw[hyperedge] (3) to[out=0,in=120,looseness=1.2] (5) to[out=120,in=170,looseness=1.2] (6) to[out=170,in=0,looseness=1.2] (3);
\draw[hyperedge] (1) to[out=10,in=170,looseness=1.5] (4) to[out=170,in=190] (6) to[out=190,in=10] (1);
\draw[hyperedge] (2) to[out=0,in=140] (4) to[out=140,in=210,looseness=1.5] (6) to[out=210,in=0] (2);
\draw[hyperedge] (3) to[out=-10,in=120,looseness=1.2] (4) to[out=120,in=190,looseness=1.2] (6) to[out=190,in=-10,looseness=1.2] (3);
\end{tikzpicture}
}
   \\
   \hline
$F_5$
&
123, 124, 345
&
\vc{\begin{tikzpicture}\outercycle{6}{5}
\drawhyperedge{0}{5}
\drawhypervertex{0}{0}
\drawhypervertex{1}{0}
\drawhypervertex{2}{0}
\drawhyperedge{1}{5}
\drawhypervertex{0}{1}
\drawhypervertex{1}{1}
\drawhypervertex{3}{1}
\drawhyperedge{2}{5}
\drawhypervertex{2}{2}
\drawhypervertex{3}{2}
\drawhypervertex{4}{2}
\end{tikzpicture} 
}
&
\vc{
\begin{tikzpicture}
\path (0,0.7) -- (0,-0.7); 
\draw
(0, 0.5) coordinate(0) node[vtx,label=left:{\tiny $1$}]{}
(0, -0.5) coordinate(1) node[vtx,label=left:{\tiny $2$}]{}
(0.7, 0.5) coordinate(2) node[vtx,label=right:{\tiny $3$}]{}
(0.7,-0.5) coordinate(3) node[vtx,label=right:{\tiny $4$}]{}
(1.5, 0) coordinate(4) node[vtx,label=right:{\tiny $5$}]{}
;
\draw[hyperedge] (0) to[out=-10,in=190,looseness=1] (2) to[out=190,in=60,looseness=1] (1) to[out=60,in=-10,looseness=1] (0);
\draw[hyperedge] (1) to[out=10,in=170,looseness=1] (3) to[out=170,in=-60,looseness=1] (0) to[out=-60,in=10,looseness=1] (1);
\draw[hyperedge] (2) to[out=310,in=50,looseness=1] (3) to[out=50,in=180,looseness=1] (4) to[out=180,in=310,looseness=1] (2);
\end{tikzpicture}
} 
   \\
   \hline
 $\mathbb{F}$
 &
 123, 345, 156, 246, 147, 257, 367
 &
\vc{\begin{tikzpicture}\outercycle{8}{7}
\drawhypervertex{0}{0}
\drawhypervertex{1}{0}
\drawhypervertex{2}{0}
\drawhyperedge{1}{7}
\drawhypervertex{0}{1}
\drawhypervertex{3}{1}
\drawhypervertex{6}{1}
\drawhyperedge{2}{7}
\drawhypervertex{0}{2}
\drawhypervertex{4}{2}
\drawhypervertex{5}{2}
\drawhyperedge{3}{7}
\drawhypervertex{1}{3}
\drawhypervertex{3}{3}
\drawhypervertex{5}{3}
\drawhyperedge{4}{7}
\drawhypervertex{1}{4}
\drawhypervertex{4}{4}
\drawhypervertex{6}{4}
\drawhyperedge{5}{7}
\drawhypervertex{2}{5}
\drawhypervertex{3}{5}
\drawhypervertex{4}{5}
\drawhyperedge{6}{7}
\drawhypervertex{2}{6}
\drawhypervertex{5}{6}
\drawhypervertex{6}{6}
\end{tikzpicture} 
}
&
\vc{
\begin{tikzpicture}
\path (60:2.2) -- (0,-0.3); 
\draw
(60:2) coordinate(1) coordinate(7) node[vtx]{}
(60:1) coordinate(2) coordinate(8) node[vtx]{}
(0,0) coordinate(3) node[vtx]{}
(1,0) coordinate(4) node[vtx]{}
(2,0) coordinate(5) node[vtx]{}
++ (120:1) coordinate(6) node[vtx]{}
(30:1.2) coordinate(X) node[vtx]{}
;
\foreach \i in {1,3,5}{
\pgfmathtruncatemacro{\j}{\i+1}
\pgfmathtruncatemacro{\k}{\i+2}
\pgfmathtruncatemacro{\l}{\i+3}
\draw[hyperedge] (\i) to[bend right=10] (\j) to[bend right=10] (\k) to[bend left=10] (\i) ;

\draw[hyperedge] (\i) to[bend right=10] (X) to[bend right=10] (\l) to[bend left=10] (\i) ;

}

\draw[hyperedge] (2) to[bend right=40] (4) to[bend right=40] (6) to[bend left=80,looseness=2.5] (2) ;

\draw
\foreach \i in {1,2,3,4,5,6,X}{
(\i) node[vtx]{}
}
;

\end{tikzpicture}
} 
   \\
   \hline
$C_5$
&
123, 234, 345, 145, 125
&
\vc{\begin{tikzpicture}\outercycle{6}{5}
\drawhyperedge{0}{5}
\drawhypervertex{0}{0}
\drawhypervertex{1}{0}
\drawhypervertex{2}{0}
\drawhyperedge{1}{5}
\drawhypervertex{1}{1}
\drawhypervertex{2}{1}
\drawhypervertex{3}{1}
\drawhyperedge{2}{5}
\drawhypervertex{2}{2}
\drawhypervertex{3}{2}
\drawhypervertex{4}{2}
\drawhyperedge{3}{5}
\drawhypervertex{0}{3}
\drawhypervertex{3}{3}
\drawhypervertex{4}{3}
\drawhyperedge{4}{5}
\drawhypervertex{0}{4}
\drawhypervertex{1}{4}
\drawhypervertex{4}{4}
\end{tikzpicture} 
}
&
\vc{
\begin{tikzpicture}
\path (0,1.2) -- (0,-1.1); 
\draw
\foreach \i in {0,1,...,7}{
(90+72*\i:1) coordinate(\i) node[vtx]{}
}
;
\foreach \i in {0,1,2,3,4}{
\pgfmathtruncatemacro{\j}{\i+1}
\pgfmathtruncatemacro{\k}{\i+2}
\draw[hyperedge] (\i) to[out=230+72*\i,in=270+72*\j,looseness=1] (\j) to[out=270+72*\j,in=310+72*\k,looseness=1] (\k) to[out=310+72*\k,in=230+72*\i,looseness=1] (\i);
}
\end{tikzpicture}
} 
   \\
   \hline
$C_5^-$
&
123, 234, 345, 145
&
\vc{\begin{tikzpicture}\outercycle{6}{5}
\drawhyperedge{0}{5}
\drawhypervertex{0}{0}
\drawhypervertex{1}{0}
\drawhypervertex{2}{0}
\drawhyperedge{1}{5}
\drawhypervertex{1}{1}
\drawhypervertex{2}{1}
\drawhypervertex{3}{1}
\drawhyperedge{2}{5}
\drawhypervertex{2}{2}
\drawhypervertex{3}{2}
\drawhypervertex{4}{2}
\drawhyperedge{3}{5}
\drawhypervertex{0}{3}
\drawhypervertex{3}{3}
\drawhypervertex{4}{3}
\end{tikzpicture} 
}
&
\vc{
\begin{tikzpicture}
\path (0,1.2) -- (0,-1.1); 
\draw
\foreach \i in {0,1,...,7}{
(90+72*\i:1) coordinate(\i) node[vtx]{}
}
(0) node[left]{\tiny $1$}
(1) node[left]{\tiny $2$}
(2) node[left]{\tiny $3$}
(3) node[right]{\tiny $3$}
(4) node[right]{\tiny $3$}
;
\foreach \i in {0,1,2,3}{
\pgfmathtruncatemacro{\j}{\i+1}
\pgfmathtruncatemacro{\k}{\i+2}
\draw[hyperedge] (\i) to[out=230+72*\i,in=270+72*\j,looseness=1] (\j) to[out=270+72*\j,in=310+72*\k,looseness=1] (\k) to[out=310+72*\k,in=230+72*\i,looseness=1] (\i);
}
\end{tikzpicture}
} 
\\
   \hline
$\overline{K_5^-}$
&
123
&
\vc{\begin{tikzpicture}\outercycle{6}{5}
\drawhyperedge{0}{5}
\drawhypervertex{0}{0}
\drawhypervertex{1}{0}
\drawhypervertex{2}{0}
\end{tikzpicture} 
}
&
\vc{
\begin{tikzpicture}
\path (0,1.2) -- (0,-1.1); 
\draw
\foreach \i in {0,1,...,7}{
(90+72*\i:1) coordinate(\i) node[vtx]{}
}
(0) node[left]{\tiny $1$}
(1) node[left]{\tiny $2$}
(2) node[left]{\tiny $3$}
(3) node[right]{\tiny $3$}
(4) node[right]{\tiny $3$}
;
\foreach \i in {0}{
\pgfmathtruncatemacro{\j}{\i+1}
\pgfmathtruncatemacro{\k}{\i+2}
\draw[hyperedge] (\i) to[out=230+72*\i,in=270+72*\j,looseness=1] (\j) to[out=270+72*\j,in=310+72*\k,looseness=1] (\k) to[out=310+72*\k,in=230+72*\i,looseness=1] (\i);
}
\end{tikzpicture}
} 
   \\
   \hline   
$\overline{K_5^=}$
&
123, 345
&
\vc{\begin{tikzpicture}\outercycle{6}{5}
\drawhyperedge{0}{5}
\drawhypervertex{0}{0}
\drawhypervertex{1}{0}
\drawhypervertex{2}{0}
\drawhyperedge{1}{5}
\drawhypervertex{2}{1}
\drawhypervertex{3}{1}
\drawhypervertex{4}{1}
\end{tikzpicture} 
}
&
\vc{
\begin{tikzpicture}
\path (0,1.2) -- (0,-1.1); 
\draw
\foreach \i in {0,1,...,7}{
(90+72*\i:1) coordinate(\i) node[vtx]{}
}
(0) node[left]{\tiny $1$}
(1) node[left]{\tiny $2$}
(2) node[left]{\tiny $3$}
(3) node[right]{\tiny $3$}
(4) node[right]{\tiny $3$}
;
\foreach \i in {0,2}{
\pgfmathtruncatemacro{\j}{\i+1}
\pgfmathtruncatemacro{\k}{\i+2}
\draw[hyperedge] (\i) to[out=230+72*\i,in=270+72*\j,looseness=1] (\j) to[out=270+72*\j,in=310+72*\k,looseness=1] (\k) to[out=310+72*\k,in=230+72*\i,looseness=1] (\i);
}
\end{tikzpicture}
} 
   \\
   \hline
$\overline{K_5^<}$
&
123, 234
&
\vc{\begin{tikzpicture}\outercycle{6}{5}
\drawhyperedge{0}{5}
\drawhypervertex{0}{0}
\drawhypervertex{1}{0}
\drawhypervertex{2}{0}
\drawhyperedge{1}{5}
\drawhypervertex{1}{1}
\drawhypervertex{2}{1}
\drawhypervertex{3}{1}
\end{tikzpicture} 
}
&
\vc{
\begin{tikzpicture}
\path (0,1.2) -- (0,-1.1); 
\draw
\foreach \i in {0,1,...,7}{
(90+72*\i:1) coordinate(\i) node[vtx]{}
}
(0) node[left]{\tiny $1$}
(1) node[left]{\tiny $2$}
(2) node[left]{\tiny $3$}
(3) node[right]{\tiny $3$}
(4) node[right]{\tiny $3$}
;
\foreach \i in {0,1}{
\pgfmathtruncatemacro{\j}{\i+1}
\pgfmathtruncatemacro{\k}{\i+2}
\draw[hyperedge] (\i) to[out=230+72*\i,in=270+72*\j,looseness=1] (\j) to[out=270+72*\j,in=310+72*\k,looseness=1] (\k) to[out=310+72*\k,in=230+72*\i,looseness=1] (\i);
}
\end{tikzpicture}
} 
   \\
   \hline    
    \caption{\label{tab:pic}Description of some $3$-uniform hypergraphs. 
    }
    \label{tab:3uniformfig}
  \end{longtable}
\end{center}

\begin{theo}
\label{F5exact}
There exists a number $n_0$ such that for all $n\geq n_0$
\begin{align*}
    \textup{exco}_2(n,F_5)=\textup{co}_2(S_n).
\end{align*}
Furthermore, $S_n$ is the unique $F_5$-free $3$-graph $G$ on $n$ vertices satisfying $\textup{co}_2(G)=\textup{exco}_2(n,F_{5})$.
\end{theo}

The \emph{loose $s$-path} $P^k_s$ is the $k$-graph with $s$ edges $\{e_1,\ldots, e_s\}$ such that $|e_i\cap e_j|=1$ if $|i-j|=1$ and $e_i\cap e_j=\emptyset$ otherwise. The \emph{loose $s$-cycle} $C^k_s$ is the $k$-uniform hypergraph with $s$ edges $\{e_1,\ldots, e_s\}$ obtained from an $(s-1)$-path $\{e_1,\ldots, e_{s-1}\}$ by adding an edge $e_s$ that shares one vertex with $e_1$, another vertex with $e_{s-1}$ and is disjoint from the other edges.

\begin{theo}\label{codcyclepath}
Let $s\geq 4$. Then,
\begin{align*}
    \textup{exco}_2(n,C^3_s)=\left \lfloor \frac{s-1}{2}\right \rfloor n^3(1+o(1)) \quad \quad \text{and} \quad \quad \textup{exco}_2(n,P^3_s)=\left \lfloor \frac{s-1}{2}\right \rfloor n^3(1+o(1)).
\end{align*}
\end{theo}
Denote by $M_s^3$ the \emph{$3$-uniform matching} of size $s$, i.e., the $3$-uniform hypergraph on $3s$ vertices with $s$ pairwise disjoint edges. 
\begin{theo}\label{codmatching}
Let $s\geq 2$. Then,
\begin{align*}
    \textup{exco}_2(n,M_s^3)=(s-1)n^3(1+o(1)).
\end{align*}
\end{theo}
Denote by $S_s^3$ the \emph{star} with $s$ edges such that the intersection of any pair of edges is exactly the same vertex. 
\begin{theo}
\label{codstar}
Let $s\geq 3$. If $s$ is odd, then
\begin{align*}
    \textup{exco}_2(n,S_s^3)=s(s-1)n^2 (1+o(1)).
\end{align*}
If $s$ is even, then
\begin{align*}
    \textup{exco}_2(n,S_s^3)=\left(s^2-\frac{3}{2}s \right)n^2 (1+o(1)).
\end{align*}
\end{theo}

In this work, we are not intending to duplicate or replace the excellent survey~\cite{Keevashsurvey} by Keevash; our aim is to supplement it with new results and directions.

Our paper is organized as follows; in Section~\ref{prepa} we explain our notation and explain the flag algebra technique briefly.
In Section~\ref{boundsfortable} we present the constructions leading to the bounds on $\sigma,\pi,\pi_u,\pi_2$ in Table~\ref{tab:exactcod2} and state conjectures on some of the non-sharp results.
In Section~\ref{F5section} we present the proof of Theorem~\ref{F5exact}. In Sections~\ref{sec:cycle}, \ref{sec:matching} and \ref{sec:star} we prove Theorems~\ref{codcyclepath}, \ref{codmatching} and \ref{codstar}, respectively.  Finally, we present related questions for higher uniformities in Section~\ref{openques}. 

\section{Preliminaries}
\label{prepa}

\subsection{Terminology and notation}
Let $H$ be a $3$-uniform hypergraph, $x,y,z\in V(H)$ and $A,B,C\subseteq V(H)$ be pairwise disjoint sets. We will use the following notation throughout the paper.
\begin{itemize}[leftmargin=*]
\item For an edge $e=\{x,y,z\}\in E(H)$ we write $xyz$ for convenience. 
    \item Denote by $L(x)$ the link graph of $x$, i.e., the graph on $V(H)\setminus \{x\}$ with $ab\in E(L(x))$ iff $abx\in E(H)$. 
    \item Denote by $L_A(x)=L(x)[A]$ the induced link graph on $A$. 
    \item Denote by $L_{A,B}(x)$ the subgraph of the link graph of $x$ only containing edges between $A$ and $B$, i.e., $V(L_{A,B}(x))=V(H)\setminus\{x\}$ and $ab\in E(L_{A,B}(x))$ iff $a\in A, b\in B$ and $abx\in E(H)$.
    \item Denote by $L^c_{A,B}(x)$ the subgraph of the link graph of $x$ only containing non-edges between $A$ and $B$, i.e., $V(L_{A,B}(x))=V(H)\setminus\{x\}$ and $ab\in E(L^c_{A,B}(x))$ iff $a\in A, b\in B$ and $abx \not\in E(H)$.
    \item $e(A,B)$ denotes the number of cross-edges between $A$ and $B$, this means \\ $e(A,B):=|\{xyz\in E(H): x,y\in A, z\in B \}|+ |\{xyz\in E(H): x,y\in B, z\in A \}|.$
    \item Denote $e(A,B,C)$ the number of cross-edges between $A,B$ and $C$, i.e.,\\ $e(A,B,C):=|\{abc\in E(H): a\in A, b\in B,c\in C \}|$. 
    \item Let $e=xyz\in E(H)$ be an edge. Define the weight of $e$ to be
    \[
    w_H(e)=d(x,y)+d(x,z)+d(y,z).
    \]
    We drop the index if the hypergraph $H$ is clear from the context. 
    \item The shadow graph $\delta H$ is the subset of all two element subsets of $V(H)$ that contains all pairs contained in some $e\in E(H)$.
\end{itemize}
\begin{defn}
Let $H$ be a $k$-graph and $t\in \mathbb{N}$. The \emph{blow-up} $H(t)$ of $H$ is the $k$-graph obtained by replacing each vertex $x\in V(H)$ by $t$ vertices $x^1,\ldots,x^t$ and each edge $x_1 \cdots x_k\in E(H)$ by $t^k$ edges $x_1^{a_1}\cdots x_k^{a_k}$ with $1 \leq  a_1,\ldots, a_k \leq t$. 
\end{defn}

\subsection{The uniform Tur\'an density}
\label{sec:uniformturan}
Recently, Reiher, R\"{o}dl and Schacht~\cite{MR3764068} introduced the uniform Tur\'an density. 
They characterised $3$-graphs $H$ satisfying $\pi_u(H)=0$.
\begin{theo}[Reiher, R\"{o}dl and Schacht~\cite{MR3764068}]
\label{uniformturan}
For a $3$-uniform hypergraph $H$, the following are equivalent:
\begin{itemize}
    \item[(a)] $\pi_u(H)=0$.
    \item[(b)] There is a labeling of the vertex set $V(H)=\{v_1,\ldots,v_f\}$ and there is a three-colouring $\phi: \delta F \rightarrow \{\text{red, blue, green}\}$ of the pairs of vertices covered by hyperedges of $H$ such that every hyperedge $v_iv_jv_k\in E(H)$ with $i<j<k$ satisfies
\begin{align*}
    \phi(v_i,v_j)=
\text{red}, \quad     \phi(v_i,v_k)=
\text{blue}, \quad     \phi(v_j,v_k)=
\text{green}.
\end{align*}

\end{itemize}
\end{theo}
We will use their result to observe that $\pi_u(H)=0$ for $H=F_{3,2}, F_5$ and $\mathbb{F}$ in Section~\ref{boundsfortable}. Reiher, R\"{o}dl and Schacht~\cite{MR3764068} also observed that the uniform Tur\'an density has a jump from $0$ to $1/27$. 

\begin{corl}[Reiher, R\"{o}dl and Schacht~\cite{MR3764068}]
\label{uniformturanjump}
If a $3$-graph $H$ satisfies $\pi_u(H)>0$, then $\pi_u(H)\geq \frac{1}{27}$.
\end{corl}
Somewhat surprisingly, it was not easy to prove that $1/27$ is best possible, which was done by Garbe, Kr\'al and Lamaison~\cite{kral}.
\subsection{Flag Algebras}

\label{flagresults}
We use flag algebras to obtain upper bounds on $\pi,\sigma,\pi_2$ and $\pi_u$ for Table~\ref{tab:exactcod2}. The general framework of flag algebras was developed by Razborov~\cite{RazbarovK43}.
Flag algebras has been applied to variety of problems, including problems on $3$-graphs~\cite{BaberTalbot,BaberTuran,RazbarovK43}. By now, it is a standard application to obtain some upper bounds on $\pi$. For $\sigma,\pi_2$ and $\pi_u$ we give a short explanation how upper bounds can be obtained. 
For readers familiar with flag algebras, the function counting scaled codegree squared sum can be expressed using flag algebras as follows:

\begin{align}\label{eq:fa}
\lim_{n \to \infty} \frac{\textup{co}_2(G_n)}{\binom{n}{2}(n-2)^2}
=
\left\llbracket
\left(
\vc{
\begin{tikzpicture}
\draw
(0,0) coordinate(1) node[vtxl,label=below:1](a){}
(1,0) coordinate(2) node[vtxl,label=below:2](b){}
(0.5,1) coordinate(3) node[vtx](c){}
;
\draw[hyperedge] (1) to[out=40,in=130] (2) to[out=120,in=280] (3) to[out=260,in=50] (1);
\draw
(1) node[vtxl]{}
(2) node[vtxl]{}
;
\end{tikzpicture}
}
\right)^2
\right\rrbracket_{1,2}
=
\frac{1}{6}
\vc{
\begin{tikzpicture}
\draw
(0,0) coordinate(1) node[vtx](a){}
(1,0) coordinate(2) node[vtx](b){}
(1,1) coordinate(3) node[vtx](c){}
(0,1) coordinate(4) node[vtx](d){}
;
\draw[hyperedge] (1) to[out=30,in=90] (2) to[out=100,in=260] (3) to[out=260,in=30] (1);
\draw[hyperedge] (1) to[out=90,in=150] (2) to[out=150,in=280] (4) to[out=280,in=90] (1);
\end{tikzpicture}
}
+
\frac{1}{2}
\Kfourthreeminus
+
\Kfourthree
,
\end{align}
 where $(G_n)_{n \geq 1}$ is a convergent sequence with $G_n$ being an $n$-vertex $3$-graph. 
 For a well-written explanation of the flag algebra method in the setting of $3$-uniform hypergraphs see \cite{RavryTuran} and for the particular application and an explanation of \eqref{eq:fa} see \cite{BalCleLidmain}. 
 
 For the minimum codegree threshold $\pi_2$, we use the formulation from \cite{FalgasK4-}. Suppose we want to show that $\pi_2(H) \leq z$ for some $3$-graph $H$ and $z \in (0,1)$.
 We do this by arguing that there is no convergent sequence $(G_n)_{n \geq 1}$ of $H$-free graphs with minimum codegree $z$. Suppose there is such a sequence. Then for any flag $F$ with two labeled vertices
 \begin{align}\label{eq:fa2}
\left\llbracket
\left(
\vc{
\begin{tikzpicture}
\draw
(0,0) coordinate(1) node[vtxl,label=below:1](a){}
(1,0) coordinate(2) node[vtxl,label=below:2](b){}
(0.5,1) coordinate(3) node[vtx](c){}
;
\draw[hyperedge] (1) to[out=40,in=130] (2) to[out=120,in=280] (3) to[out=260,in=50] (1);
\draw
(1) node[vtxl]{}
(2) node[vtxl]{}
;
\end{tikzpicture}
}
-z
\right)
\times F
\right\rrbracket_{1,2}
\geq 0.
\end{align}
By combining these inequalities and sum of squares inequalities, one gets a contradiction with the existence of $(G_n)_{n \geq 1}$. 

It is not obvious to us how to express in the flag algebras language the conditions for uniform Tur\'an density. Instead of a direct expression, Glebov, Kr\'al', and Volec~\cite{MR3474967} formulated consequences of the conditions that can be expressed using flag algebras. 
Suppose graphs in a convergent sequence $(G_n)_{n \geq 1}$ are $(d,\eta,1)$-dense.
The simplest instance is to take two labeled vertices and look at their co-neighborhood $N$. In $N$, the relative density of edges must be at least $d$ if $N$ is sufficiently large. If it is not sufficiently large, then the values in the equation below become all $0$ anyway.
In flag algebras, this can be done as 
 \begin{align}\label{eq:fau}
\left\llbracket
\left(
(1-d)
\vc{
\begin{tikzpicture}
\draw
(0,0) coordinate(1) node[vtxl,label=below:1](a){}
(1,0) coordinate(2) node[vtxl,label=below:2](b){}
(0,1) coordinate(3) node[vtx](c){}
(0.5,1) coordinate(4) node[vtx](c){}
(1,1) coordinate(5) node[vtx](c){}
;
\draw[hyperedge] (1) to[out=40,in=130] (2) to[out=120,in=280] (3) to[out=260,in=50] (1);
\draw[hyperedge] (1) to[out=40,in=130] (2) to[out=120,in=280] (4) to[out=260,in=50] (1);
\draw[hyperedge] (1) to[out=40,in=130] (2) to[out=120,in=280] (5) to[out=260,in=50] (1);
\draw[hyperedge] (3) to[out=40,in=130] (4) to[out=40,in=130] (5) to[out=90,in=90] (3);
\draw
(1) node[vtxl]{}
(2) node[vtxl]{}
(3) node[vtx](c){}
(4) node[vtx](c){}
(5) node[vtx](c){}
;
\end{tikzpicture}
}
-d
\vc{
\begin{tikzpicture}
\draw
(0,0) coordinate(1) node[vtxl,label=below:1](a){}
(1,0) coordinate(2) node[vtxl,label=below:2](b){}
(0,1) coordinate(3) node[vtx](c){}
(0.5,1) coordinate(4) node[vtx](c){}
(1,1) coordinate(5) node[vtx](c){}
;
\draw[hyperedge] (1) to[out=40,in=130] (2) to[out=120,in=280] (3) to[out=260,in=50] (1);
\draw[hyperedge] (1) to[out=40,in=130] (2) to[out=120,in=280] (4) to[out=260,in=50] (1);
\draw[hyperedge] (1) to[out=40,in=130] (2) to[out=120,in=280] (5) to[out=260,in=50] (1);
\draw
(1) node[vtxl]{}
(2) node[vtxl]{}
(3) node[vtx](c){}
(4) node[vtx](c){}
(5) node[vtx](c){}
;
\end{tikzpicture}
}
\right)
\times F
\right\rrbracket_{1,2}
\geq 0,
\end{align}
where edges with exactly one labeled vertex are not depicted for the sake of readability and $F$ is any  flag with two labeled vertices. An analogous equation can be obtained by considering non-co-neighborhood. More generally,
for more labeled vertices, we can decide for each pair if we want to take co-neighborhood or non-co-neighborhood. 
The approach by Glebov, Kr\'al', and Volec~\cite{MR3474967} continues by getting an exact solution on the threshold that has density of a particular hypergraph 0. Then to a hypothetical counterexample, one can apply sparsification to obtain an example on the threshold density, which in turn shows that the hypothetical counterexample actually has zero edge density, which is a contradiction. 
Since the full method from \cite{MR3474967} is more involved and uniform Tur\'an density is not the main focus of this paper, we opted for a simpler approach. We use \eqref{eq:fau} and its analogue for non-co-neighborhood and obtain an upper bound on the global edge density that is strictly less than $d$, leading to a contradiction. 
This simplified approach is not as strong. For $K_4^{3-}$ it gives only $0.259$ while a sharp result 0.25 is known~\cite{MR3474967}.

 The calculations which lead to the results in Table~\ref{tab:exactcod2} are computer assisted; we use CSDP~\cite{csdp} to calculate numerical solutions of semidefinite programs and then use SageMath~\cite{sagemath} for rounding these numerical solutions to exact ones.
The data files and programs to reproduce the calculations we developed are available at \url{http://lidicky.name/pub/co2b/}.\\
Next, we will present the constructions which give the lower bounds from Table~\ref{tab:exactcod2}.

\section{Bounds from Table~\ref{tab:exactcod2}}
\label{boundsfortable}

\subsection{\texorpdfstring{$K_4^3$}{TEXT}}
\label{K43}
Tur\'an's tetrahedron problem asks to determine the Tur\'an density of $K_4^3$. The best lower bound $\pi(K_4^3)\geq 5/9$ is obtained by $C_n$, see Figure~\ref{fig:CnBnH5}, the $3$-graph on $n$ vertices with vertex set $V(C_n)=V_1\cup V_2 \cup V_3$ where $||V_i|-|V_j||\leq 1$ for $i\neq j$ and edge set 
\begin{align*}
    E(C_n)=\{abc: a\in V_1,b\in V_2, c\in V_3 \} \cup \{abc: a,b\in V_1,c\in V_2 \} \\ \cup \ \{abc: a,b\in V_2,c\in V_3 \} \cup  \{abc: a,b\in V_3,c\in V_1 \}.
    \end{align*}
 Brown~\cite{K43brown}, Kostochka~\cite{K43Kostochka}, Fon-der-Flaass~\cite{K43Fonderflaass} and Frohmader~\cite{MR2465761} constructed families of $K_4^3$-free $3$-graphs with the same number of edges.  A series of papers~\cite{TetrahedronCaen,ChungLutetrahedron,RazbarovK43} have improved on the upper bound of $\pi(K_4^3)$ culminating in the current best-known bound by Baber~\cite{BaberTuran}
 \begin{align*}
     \pi(K_4^3)\leq 0.5615.
 \end{align*}
We~\cite{BalCleLidmain} solved Tur\'an's tetrahedron problem for the codegree squared density by showing $\sigma(K_4^3)=1/3$, where the lower bound is achieved by $C_n$. \\
For the minimum codegree threshold Czygrinow and Nagle~\cite{MR1829685} provided a construction that shows $\pi_2(K_4^3)\geq 1/2$. Let $T$ be a uniformly at random chosen tournament on $n$ vertices. Define a $3$-graph $G_T$ on $n$ vertices by setting  the triple $ijk$ with $i < j < k$ to be an edge of $G_T$ if the ordered pairs $(i,j)$ and $(i,k)$ receive opposite directions in $T$. This $3$-graph is $K_4^3$-free, and has minimum codegree $n/2-o(n)$ with high probability. Using flag algebras we can prove that $\pi_2(K_4^3)\leq 0.529$.

The previous construction was used by R\"{o}dl~\cite{MR837962} to show that for the uniform Tur\'an density $\pi_u(K_4^3)\geq 1/2$. Erd\H{o}s~\cite{MR1083590} suggested that this might be best possible, also see Reiher~\cite{MR4111729}. Using flag algebras as described in Section~\ref{flagresults}, we obtain $\pi_u(K_4^3) \leq 0.531$.

\subsection{\texorpdfstring{$K_5^3$}{TEXT}}
\label{K53}
For the Tur\'an density of $K_5^3$ it is known that
\begin{align*}
    \frac{3}{4}\leq \pi(K_5^3)\leq 0.769533,
\end{align*}
where the upper bound is obtained via flag algebras~\cite{flagmatic} and the lower bound is obtained by the balanced complete bipartite 3-graph $B_n$ (see Figure~\ref{fig:CnBnH5}) as observed by Tur\'an~\cite{MR177847}. The following $n$-vertex $3$-graph $H_5$ (presented in \cite{Sidorenko1981SystemsOS}, see Figure~\ref{fig:CnBnH5}) also achieves the lower bound. 
The vertex set of $H_5$ is divided into $4$ parts $A_1, A_2, A_3, A_4$ with $||A_j|-|A_i||\leq 1$ for all $1\leq i\leq j\leq 4$ and a triple $e$ is not an edge of $H_5$ iff there is some $j$ ($1\leq j \leq 4$) such that $|e \cap A_{j}|\geq 2$ and $|e\cap A_{j}|+|e\cap A_{j+1}|=3$, where $A_{5}=A_1$.

We~\cite{BalCleLidmain} proved that $\sigma(K_5^3)=5/8$, where the lower bound is obtained by the complete balanced bipartite $3$-graph $B_n$. Notice that $H_5$ does not achieve this lower bound. 

For the minimum codegree threshold, we have 
\begin{align*}
\frac{2}{3} \leq \pi_2(K_5^3)\leq 0.74,
\end{align*}
where we obtained the upper bound via flag algebras and the lower bound is due to Falgas-Ravry~\cite{codegreeFalgas} and Lo and Markstr\"om~\cite{LoMark} who constructed lower bounds on the codegree threshold for cliques of arbitrary size $s$. Here we present the construction due to Falgas-Ravry. Let $c: E(K_n)\rightarrow [s]$ be a uniformly at random chosen colouring of the edges of the complete graph on $n$ vertices with $s$ colours. Consider the $3$-graph $G$ based on this colouring in the following way: A triple $i,j,k$ with $i < j < k$ forms an edge in $G$ if and only if $c(ij)\neq c(ik)$. This $3$-graph is $K_{s+2}$-free and has minimum codegree $(1-1/s)n-o(n)$ with high probability. This construction can be seen as a generalization of the Czygrinow and Nagle  construction from the previous section.

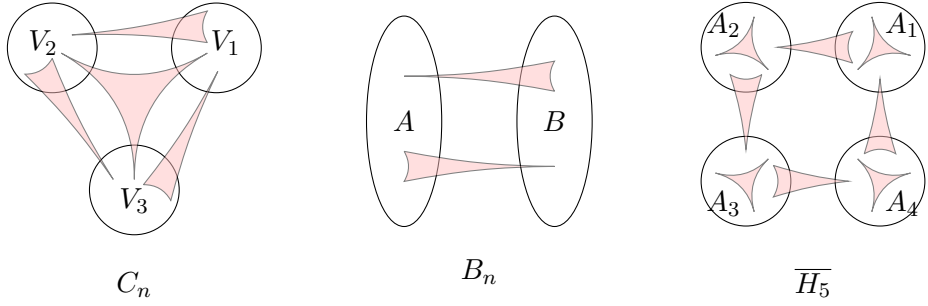
\begin{figure}
\begin{center}
\tikzset{
vtx/.style={inner sep=1.1pt, outer sep=0pt, circle, fill,draw}, 
hyperedge/.style={fill=pink,opacity=0.5,draw=black}, 
vtxBig/.style={inner sep=12pt, outer sep=0pt, circle, fill=white,draw}, 
hyperedge/.style={fill=pink,opacity=0.5,draw=black}, 
}
\vc{
\begin{tikzpicture}[scale=1.4]
\draw (30:0.9) coordinate(x1) node[vtxBig]{};
\draw (150:0.9) coordinate(x2) node[vtxBig]{};
\draw (270:0.9) coordinate(x3) node[vtxBig]{};
\draw
(30:0.8) coordinate(1) 
(150:0.8) coordinate(2) 
(270:0.8) coordinate(3) 
;
\draw[hyperedge] (1) to[out=210,in=330] (2) to[out=330,in=90] (3) to[out=90,in=210] (1);
\foreach \ashift in {30,150,270}{
\draw
(\ashift:0.8)++(\ashift+60:0.1) coordinate(1) 
++(\ashift+60:0.3) coordinate(2) 
(\ashift+120:0.8)++(\ashift+120+90:-0.2) coordinate(3)
;
\draw[hyperedge] (1) to[bend left] (2) to[bend left=5] (3) to[bend left=5] (1);
}
\draw (30:1)  node{$V_1$};
\draw (150:1) node{$V_2$};
\draw (270:1) node{$V_3$};
\draw (0,-1.8) node{$C_n$};
\end{tikzpicture}
}
\hskip 1cm
\vc{
\begin{tikzpicture}[scale=2.0]
\draw (0,0) coordinate(x1) ellipse (0.25cm and 0.7cm);
\draw (1,0) coordinate(x1) ellipse (0.25cm and 0.7cm);
\draw
(0,-0.2) coordinate(1) 
(0,-0.4) coordinate(2) 
(0,0.3) coordinate(3) 
(1,0.2) coordinate(4) 
(1,-0.3) coordinate(5) 
(1,0.4) coordinate(6) 
;
\draw[hyperedge] (1) to[bend left] (2) to[bend left=5] (5) to[bend left=5] (1);
\draw[hyperedge] (4) to[bend left] (6) to[bend left=5] (3) to[bend left=5] (4);
\draw (0,0) node{$A$};
\draw (1,-0) node{$B$};
\draw (0.5,-1) node{$B_n$};
\end{tikzpicture}
}
\hskip 1cm
\vc{
\begin{tikzpicture}[scale=1.4]
\draw (45:0.9) coordinate(x1) node[vtxBig]{};
\draw (135:0.9) coordinate(x2) node[vtxBig]{};
\draw (225:0.9) coordinate(x3) node[vtxBig]{};
\draw (315:0.9) coordinate(x4) node[vtxBig]{};
\foreach \ashift in {45,135,225,315}{
\draw
(\ashift:0.9)
+(\ashift+135-30:0.3) coordinate(2) 
+(\ashift+135+30:0.3) coordinate(1) 
(\ashift+90:0.9)++(\ashift-45:0.3) coordinate(3)
;
\draw[hyperedge] (1) to[bend left] (2) to[bend left=5] (3) to[bend left=5] (1);
\draw
(\ashift:0.9)
+(\ashift+180:0.3) coordinate(1) 
+(\ashift+180+120:0.3) coordinate(2) 
+(\ashift+180+240:0.3) coordinate(3) 
;
\draw[hyperedge] (1) to[bend left] (2) to[bend left] (3) to[bend left] (1);
}
\draw (45:1.2)  node{$A_1$};
\draw (135:1.2) node{$A_2$};
\draw (225:1.2) node{$A_3$};
\draw (315:1.2) node{$A_4$};
\draw (0,-1.6) node{$\overline{H_5}$};
\end{tikzpicture}
}
\end{center}
\caption{Illustration of $C_n$, $B_n$ and the complement of $H_5$.}\label{fig:CnBnH5}
\end{figure}

For the uniform Tur\'an density, we have
\begin{align*}
    \frac{2}{3}\leq \pi_u(K_5^3) \leq 0.758,
\end{align*}
where the lower bound is obtained from the previously described $3$-graph $G$ and the upper bound from flag algebras.

\subsection{\texorpdfstring{$K_6^3$}{TEXT}}
\label{K63}
For the Tur\'an density of $K_6^3$ we know 
\begin{align*}
    0.84 \leq \pi(K_6^3)\leq 0.8583903,
\end{align*}
where the upper bound was obtained via flag algebras~\cite{flagmatic} and the lower bound (see \cite{Sidorenko1981SystemsOS}) is obtained by the following $3$-graph $H_{6}$ (see Figure~\ref{fig:G6H6}). Divide the vertex set $[n]$ of $H_6$ into $5$ parts $A_1,\ldots,A_{5}$ with $||A_j|-|A_i||\leq 1$ for all $1\leq i\leq j\leq 5$ and let a triple $e$ not form an edge of $H_{6}$ iff there is some $j$ ($1\leq j \leq 5$) such that
\begin{align*}
     |e \cap A_{j}|\geq 2 \quad \quad \text{and} \quad \quad |e\cap A_{j}|+|e\cap A_{j+1}|=3, 
\end{align*}
where $A_{6}=A_1$.  
For the codegree squared density we have
\begin{align*}
    0.7348 \leq\sigma(K_6^3)\leq 0.7535963,
\end{align*}
where we obtained the upper bound by flag algebras and the lower bound is achieved by the following construction from \cite{BalCleLidmain}. Let $G_6$ (see Figure~\ref{fig:G6H6}) be a $3$-graph with a vertex partition $A\cup B_1\cup B_2\cup B_3$ such that $|B_1|=|B_2|=|B_3|=bn$ and $|A|=an$, where $a$ and $b$ are optimized later. The $3$-graph $G_6$ forms a $C_{3bn}$ (see Subsection~\ref{K43}) on $B_1 \cup B_2 \cup B_3$ and further contains all edges intersecting $A$ with exactly $1$ or $2$ vertices. It has codegree squared sum 
\begin{align*}
    \textup{co}_2(G_6)&= \left(\frac{a^2}{2}(1-a)^2+3ab +3 \frac{b^2}{2}(a+b)^2+3b^2 (a+2b)^2\right) n^4(1+o(1))\\
    &= \left(\frac{11}{18}a^4-a^3-\frac{2}{3}a^2+\frac{8}{9}a+\frac{1}{6} \right)n^4(1+o(1)).
\end{align*}
The optimum is one of the roots of $22a^3-27a^2-12a+8=0$. The root is approximately  $a=0.412498$, which gives $\textup{co}_2(G_6)= 0.3674 n^4 (1+o(1))$. Note that this $3$-graph with $a=2/5$ is an instance of a family of constructions due to Keevash and Mubayi~\cite{Keevashsurvey} that achieves the best-known lower bound for the Tur\'an density of $K_6^3$.

For the minimum codegree threshold of $K_6^3$ we have
\begin{align*}
    \frac{3}{4} \leq \pi_2(K_6^3)\leq 0.838,
\end{align*}
where the upper bound is obtained via flag algebras and the lower bound is due to Falgas-Ravry~\cite{codegreeFalgas} and Lo and Markstr\"om~\cite{LoMark}, see the previous section for the construction by Falgas-Ravry.

For the uniform Tur\'an density, we have
\begin{align*}
    \frac{3}{4}\leq \pi_u(K_6^3) \leq 0.853,
\end{align*}
where the lower bound is obtained by the Falgas-Ravry's construction stated in the previous section and the upper bound from flag algebras.

\subsection{$K_t^k$}
Denote by $K_t^k$ the complete $k$-graph on $t$ vertices. Large cliques are studied intensively for all notions of extremality. For an overview of results in $\ell_1$-norm see \cite{MR1341481}, in $\ell_2$-norm see \cite{BalCleLidmain}, and for the minimum codegree threshold see \cite{AllanZhao} and \cite{Sidorenkocode}. 

The best-known bounds for the Tur\'an density are
\begin{align*}1-\left(\frac{k-1}{t-1}\right)^{k-1} \leq \pi(K_t^k)\leq 1-\binom{t-1}{k-1}^{-1},
\end{align*}
where the lower bound is due to Sidorenko~\cite{Sidorenko1981SystemsOS} and the upper bound is due to de Caen~\cite{MR734038}. 

For $k=3$, Keevash and Mubayi~\cite{Keevashsurvey} constructed a family of $3$-graphs obtaining this lower bound. Their construction for $K_t^3$-free graphs goes as follows. Take a directed graph $G$ on $t-1$ vertices with both in-degree and out-degree equal to one, i.e. $G$ is a vertex disjoint union of directed cycles. Note that cycles of length $2$ in $G$ are allowed but loops are not. Now we construct a $K_t^3$-free hypergraph $H$ on $(t-1)n$ vertices. 
Partition $V(H)$ into $A_1,\ldots,A_{t-1}$ of equal sizes corresponding to vertices $v_1,\ldots,v_{t-1}$ of $G$.
A triple $xyz$ is not an edge in $H$ if and only if there exists $A_i$ such that $x,y,z \in A_i$, or there exists $v_iv_j \in E(G)$ with $|\{x,y,z\} \cap A_i|=2$ and  $|\{x,y,z\} \cap A_j|=1$.
Notice that $H_5$ from Figure~\ref{fig:CnBnH5} is a case of this construction with $G$ being a directed cycle on four vertices with edges $v_1v_2, v_2v_3, v_3v_4, v_4v_1$.
Other examples are $C_n, B_n$ and $H_6$.

The best-known extremal constructions in $\ell_2$-norm for small $t$ are possibly unbalanced versions of this construction. 
\begin{itemize}
    \item For $t=4$, $C_n$ is the resulting 3-graph when $G$ is a directed triangle. 
    \item For $t=5$, $B_n$ is the resulting 3-graph when $G$ is the union of two cycles of length 2.
    \item For $t=6$, $G_6$ is the resulting 3-graph when $G$ is formed by the union of a directed triangle with a directed cycle of length 2, up to balancing of the class sizes.
\end{itemize}
This suggests~\cite{BalCleLidmain} that when $G$ is maximizing the number of directed cycles, the resulting $3$-graph $H$ could be extremal in $\ell_2$-norm.

For the minimum codegree threshold of cliques of large size, the best-known bounds, obtained by Lo and Zhao~\cite{AllanZhao}, are 
\begin{align*}
    1- c_2 \frac{\log t}{t^{k-1}}\leq \pi_2(K_t^k) \leq 1- c_1 \frac{\log t}{t^{k-1}},
\end{align*}
where $c_1,c_2>0$ are constants depending on $k$.

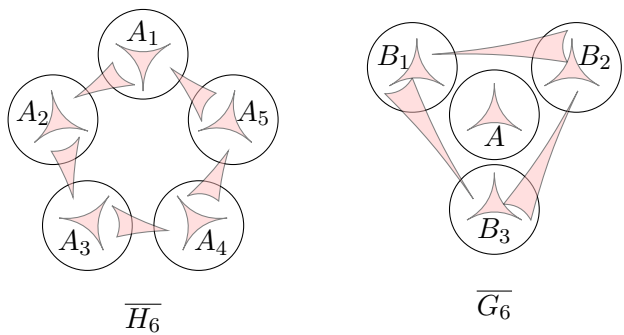
\begin{figure}
\begin{center}
\tikzset{
vtx/.style={inner sep=1.1pt, outer sep=0pt, circle, fill,draw}, 
hyperedge/.style={fill=pink,opacity=0.5,draw=black}, 
vtxBig/.style={inner sep=12pt, outer sep=0pt, circle, fill=white,draw}, 
hyperedge/.style={fill=pink,opacity=0.5,draw=black}, 
}
\vc{
\begin{tikzpicture}[scale=1.4]
\foreach \x in {1,2,3,4,5}
{
\draw (90-72+72*\x:0.9) coordinate(x\x) node[vtxBig]{};
\draw (90-72+72*\x:1.1) node{$A_{\x}$};
}
\foreach \x in {1,2,3,4,5}{
\pgfmathtruncatemacro{\ashift}{(90-72+72*\x)}
\draw
(\ashift:0.9)
+(\ashift+135-30:0.3) coordinate(2) 
+(\ashift+135+30:0.3) coordinate(1) 
(\ashift+72:0.9)++(\ashift-45:0.3) coordinate(3)
;
\draw[hyperedge] (1) to[bend left] (2) to[bend left=5] (3) to[bend left=5] (1);
\draw
(\ashift:0.85)
+(\ashift+180:0.3) coordinate(1) 
+(\ashift+180+120:0.3) coordinate(2) 
+(\ashift+180+240:0.3) coordinate(3) 
;
\draw[hyperedge] (1) to[bend left] (2) to[bend left] (3) to[bend left] (1);
}
\draw (0,-1.6) node{$\overline{H_6}$};
\end{tikzpicture}
}
\hskip 2em
\vc{
\begin{tikzpicture}[scale=1.4]
\draw (30:0.9) coordinate(x1) node[vtxBig]{};
\draw (150:0.9) coordinate(x2) node[vtxBig]{};
\draw (270:0.9) coordinate(x3) node[vtxBig]{};
\draw (0,0) coordinate(a) node[vtxBig]{};
\draw
(30:0.8) coordinate(1) 
(150:0.8) coordinate(2) 
(270:0.8) coordinate(3) 
;
\foreach \ashift in {30,150,270}{
\draw
(\ashift:0.8)++(\ashift+60:0.1) coordinate(1) 
++(\ashift+60:0.3) coordinate(2) 
(\ashift+120:0.8)++(\ashift+120+90:-0.2) coordinate(3)
;
\draw[hyperedge] (1) to[bend left] (2) to[bend left=5] (3) to[bend left=5] (1);
}
\foreach \x in {1,2,3}{
\pgfmathtruncatemacro{\ashift}{(30+120*\x)}
\draw
(\ashift:0.85)
+(\ashift+180:0.3) coordinate(1) 
+(\ashift+180+120:0.3) coordinate(2) 
+(\ashift+180+240:0.3) coordinate(3) 
;
\draw[hyperedge] (1) to[bend left] (2) to[bend left] (3) to[bend left] (1);
}
\pgfmathtruncatemacro{\ashift}{(30+120*0)}
\draw
+(\ashift+180:0.3) coordinate(1) 
+(\ashift+180+120:0.3) coordinate(2) 
+(\ashift+180+240:0.3) coordinate(3) 
;
\draw[hyperedge] (1) to[bend left] (2) to[bend left] (3) to[bend left] (1);

\draw (30:1.1)  node{$B_2$};
\draw (150:1.1) node{$B_1$};
\draw (270:1.1) node{$B_3$};
\draw (0,-0.2) node{$A$};
\draw (0,-1.8) node{$\overline{G_6}$};
\end{tikzpicture}
}
\end{center}
\caption{Illustration of the complement of $H_6$ and $G_6$.}\label{fig:G6H6}
\end{figure}

\tikzset{redge/.style={color=red, line width=1.2pt}} 
\tikzset{bedge/.style={color=blue, line width=1.2pt}} 
\tikzset{gedge/.style={color=green!80!black, line width=1.2pt}} 

\begin{figure}
\begin{center}
\begin{tikzpicture}
\draw
\foreach \x in {1,2,...,5}{
(72*\x+90-72:1) node[vtx](\x){}
(72*\x+90-72:1.3) node{\x}
}
;
\draw[redge](1)--(2) (3)--(4);
\draw[bedge](1)--(3) (1)--(4) (1)--(5) (3)--(5);
\draw[gedge](2)--(3) (2)--(4) (2)--(5) (4)--(5);
\draw(0,-2) node {$F_{3,2}$};
\begin{scope}[xshift=5cm]
\draw
\foreach \x in {1,2,...,5}{
(72*\x+90-72:1) node[vtx](\x){}
(72*\x+90-72:1.3) node{\x}
}
;
\draw[redge](1)--(2) (3)--(4);
\draw[bedge](1)--(3) (1)--(4) (3)--(5);
\draw[gedge](2)--(3) (2)--(4) (4)--(5);
\draw(0,-2) node {$F_{5}$};
\end{scope}

\begin{scope}[xshift=10cm]
\draw
\foreach \x in {1,2,...,5}{
(72*\x+90-72:1) node[vtx](\x){}
(72*\x+90-72:1.3) node{\x}
}
(72*1+90-72:1.7) node{$v_3$}
(72*2+90-72:1.7) node{$v_2$}
(72*3+90-72:1.7) node{$v_1$}
(72*4+90-72:1.7) node{$v_5$}
(72*5+90-72:1.7) node{$v_4$}
;
\draw[redge](1)--(2) (1)--(4) (3)--(4);
\draw[bedge](1)--(3) (1)--(5) (3)--(5);
\draw[gedge](2)--(3) (2)--(5) (4)--(5);
\draw(0,-2) node {$C_{5}^-$};
\end{scope}

\end{tikzpicture}    
\end{center}
    \caption{Colorings for Theorem~\ref{sec:uniformturan}(b).}
    \label{fig:colorings}
\end{figure}
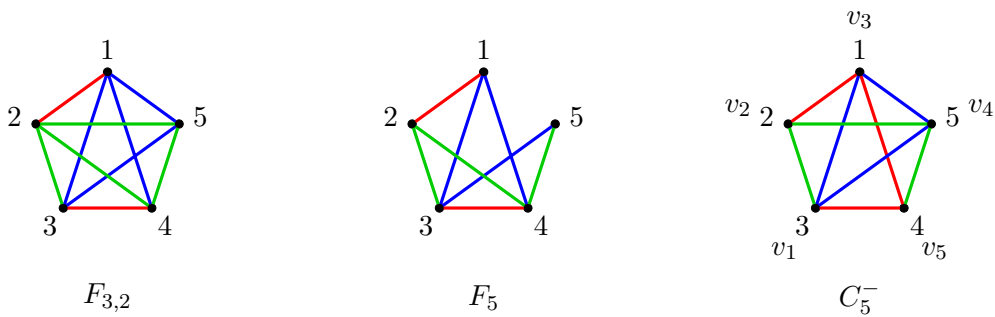

\subsection{\texorpdfstring{$F_{3,2}$}{TEXT}}
\label{F32}
F\"{u}redi, Pikhurko and Simonovits~\cite{F32furedi} proved that the Tur\'an density of $F_{3,2}$ is $4/9$ which is achieved by the $3$-graph with vertex set $V = V_1 \cup V_2$, where all edges have two vertices in $V_1$ and one vertex in $V_2$ and $|V_1| = \frac{2}{3}n$.\\
The $F_{3,2}$-free graph $G$ which we conjecture to maximize the codegree square sum has similar structure, but different class sizes: $|V_1|= \lceil \sqrt{1/2}n \rceil$. Clearly, changing the class sizes will not create an $F_{3,2}$, and $\textup{co}_2(G)=1/8 n^4 (1+o(1))$. Thus, $\sigma(F_{3,2})\geq 1/4$. Flag algebras give $\sigma(F_{3,2})\leq 1/4+10^{-9}$.
Here we note that flag algebras give a numerical result of $1/4$ but obtaining an exact solution is tricky since it contains $\sqrt{1/2}$ in addition to rational numbers. Likely, using the stability method one could prove an exact result, however we have not attempted to do it. \\
%
%
The minimum codegree threshold of $F_{3,2}$ was determined exactly by Falgas-Ravry, Marchant,  Pikhurko and Vaughan~\cite{codF32falgas} who proved that for large enough $n$
\begin{align*}
    \textup{ex}_2(n,F_{3,2})=\begin{cases} \left\lfloor \frac{n}{3}\right\rfloor -1 & \text{ if $n$ is congruent to $1$ modulo $3$,} \\ 
    \left\lfloor \frac{n}{3}\right\rfloor & \text{ otherwise.}
    \end{cases}
\end{align*}
The lower bound is obtained by the hypergraph with vertex partition $A_1\cup A_2 \cup A_3$, where $||A_i|-|A_j||\leq 1$ for $i,j\in [3]$, and edges $xyz$ for $x,y\in A_i,z\in A_{i+1}$ for $i=1,2,3$, where $A_4:=A_1$.  

The uniform Tur\'an density of $F_{3,2}$ is $\pi_u(F_{3,2})=0$ by Theorem~\ref{uniformturan}, because the edges of the shadow graph can be colored in the following way. Color the edges $12,34$ red, the edges $13,14,15,35$ blue, and the edges $23,24,25,45$ green; see Figure~\ref{fig:colorings}, and recall that the edges of $F_{3,2}$ are $123,145,245,345$.  


\subsection{\texorpdfstring{$F_{3,3}$}{TEXT}}\label{sec:F33}
Keevash and Mubayi~\cite{MR2912791} and independently Goldwasser and Hansen~\cite{MR3048207} proved that the ba\-lan\-ced, complete, bipartite $3$-graph is extremal in the $\ell_1$-norm for $n\geq 6$. Thus, $\pi(F_{3,3})=3/4$. In \cite{BalCleLidmain} we proved that it is also extremal in the $\ell_2$-norm for $n$ large enough, i.e. $\sigma(F_{3,3})=5/8$. For the minimum codegree threshold we have $1/2 \leq \pi_2(F_{3,3})\leq 0.604$, where the lower bound comes from the balanced, complete, bipartite $3$-graph $B_n$ and the upper bound is obtained via flag algebras. For the uniform Tur\'an density, we have
\begin{align}
\label{inequalityk4minus}
    \frac{1}{4}= \pi_u(K_4^{3-}) \leq \pi_u(F_{3,3}) \leq \frac{1}{4},
\end{align}
since $K_4^{3-} \subset F_{3,3}$. For $\pi_u(K_4^{3-})$ see Section~\ref{K43minus}.
The upper bound in \eqref{inequalityk4minus} was observed by Schulke~\cite{Schulke23}. 
Observe that $F_{3,3}$ can be obtained from $K_4^{3-}$ as a blow-up of one vertex and inserting an edge inside the blow-up.
At $1/4$ threshold, there exists a blow-up of $K_4^{3-}$ and by uniformity, there is an edge inside one of the parts giving a copy of $F_{3,3}$.

\subsection{\texorpdfstring{$F_5$}{TEXT}}
\label{F5}

Frankl and F\"uredi~\cite{F5Frankl} proved that for $n\geq 3000$ the $F_5$-free $3$-graph with the largest number of edges is $S_n$, thus, $\pi(F_5)=2/9$. The codegree squared density of $F_5$ is $\sigma(F_5)=2/27$ which follows from Theorem~\ref{F5exact}, see Section~\ref{F5section} for more details.

The minimum codegree threshold of $F_5$ is $\pi_2(F_5)=0$. 
To be more precise, we have $\textup{ex}_2(n,F_5)\leq 2$, because if there exists an $F_5$-free $3$-graph $G$ with $d(x,y)\geq 3$ for all pairs $xy$, then take any pair $ab$ and two vertices $c,d \in N(a,b)$. Now there exists $e\neq a,b,c,d$ with $e\in N(c,d)$. Thus $a,b,c,d,e$ spans an $F_5$, a contradiction. 

The uniform Tur\'an density of $F_5$ is $\pi_u(F_5)=0$ by Theorem~\ref{uniformturan}, because the edges of the shadow graph can be colored in the following way. Color the edges $12,34$ red, the edges $13,14,35$ blue, and the edges $23,24,45$ green; see Figure~\ref{fig:colorings}. 

\subsection{Fano plane \texorpdfstring{$\FF$}{TEXT}}
\label{Fano}
Denote by $\FF$ the Fano plane, i.e., the unique $3$-uniform hypergraph with seven edges on seven vertices in which every pair of vertices is contained in a unique edge. S\'os~\cite{VeraSos} proposed to study the Tur\'an number for the Fano plane and conjectured the complete balanced bipartite $3$-graph $B_n$ to be extremal. This problem was solved asymptotically by Caen and F\"{u}redi \cite{FanoFuredi}. Later, F\"{u}redi and Simonovits \cite{FanoFuredi2} and, independently, Keevash and Sudakov \cite{Fanoplane} determined the extremal hypergraph for large $n$. Recently, Bellmann and Reiher \cite{FanoReiher} solved the question for all $n$. \\
We conjecture that the extremal example in the codegree squared sense also is the complete bipartite graph.
\begin{conj}
There exists $n_0$, such that for all $n\geq n_0$ 
\begin{align*}
    \textup{exco}_2(n,\FF)= \textup{co}_2(B_n). 
\end{align*}
Furthermore, $B_n$ is the unique $\FF$-free $3$-graph $H$ on $n$ vertices satisfying $\textup{co}_2(H)=\textup{exco}_2(n,\FF)$.
\end{conj}
We have $5/8\leq \sigma(\FF) \leq 3/4$ as trivial bounds using the complete bipartite $3$-graph as a lower bound and the fact that $\sigma(\FF)\leq \pi(\FF)\leq 3/4$.

The minimum codegree threshold was asymptotically determined to be $\pi_2(\FF)=1/2$ by Mubayi~\cite{MubayiFano}. He conjectured that $\textup{ex}_2(n,F_{3,2})=\lfloor n/2 \rfloor$ for large enough $n$. This conjecture was solved by Keevash~\cite{MR2478542} using an involved quasi-randomness argument. Later, DeBiasio and Jiang~\cite{MR3131883} gave a simplified proof.   

We did not use flag algebras for $\FF$ because our computers cannot handle the number of $\FF$-free $3$-graphs on $7$ vertices. 

The uniform Tur\'an density of $\FF$ is $\pi_u(\FF)=0$ by Theorem~\ref{uniformturan}, because the edges of the shadow graph can be colored in the following way. Color the edges $12,34,15,24,14,25,36$ red, the edges $13,35,16,26,17,27,37$ blue, and the edges $23,45,56,46,47,57,67$ green, where the edges of the Fano plane are are $123,345,156,246,147,257,367$. 

\subsection{\texorpdfstring{$K_4^{3-}$}{TEXT}}
\label{K43minus}
Denote by $S_6$ the $3$-graph on six vertices with edge set
\begin{align*}
    E(S_6)=\{123,234,345,451,512,136,246,356,256,146\}.
\end{align*} 
Note that in $S_6$ the link graph of every vertex is a cycle of length $5$.
Let $n=6^k$ be a power of $6$ and let $G$ be the iterated blow-up of $S_6$. This $3$-graph $G$ was constructed by Frankl and F\"{u}redi~\cite{K43-extremalfrankl} to give a lower bound on the Tur\'an density of $K_4^{3-}$. Since $G$ is $K_4^{3-}$-free, $\pi(K_4^{3-})\geq 2/7$. The best know upper bound $\pi(K_4^{3-})\leq 0.28689$ is by Vaughan~\cite{flagmatic} using flag algebras.
The $3$-graph $G$ has codegree squared sum
\begin{align*}\textup{co}_2(G)&=\frac{1}{2}n^4 \left(\sum_{i\geq1} \frac{5}{6^i}\left( \frac{2}{6^i}\right)^2 + o(1) \right)=10n^4 \left(\sum_{i\geq1} \frac{1}{216^i} + o(1) \right)\\
&=10n^4 \left(\frac{1}{1-\frac{1}{216}}-1\right) (1+o(1))
=\frac{2}{43}n^4(1+o(1)).
\end{align*}
Thus, $\sigma(K_4^{3-})\geq 4/43$. We conjecture that $G$ is the extremal example in $\ell_2$-norm.
\begin{conj}
\begin{align*}\sigma(K_4^{3-})= \frac{4}{43} \approx 0.0930232558.
\end{align*}
\end{conj}
Flag algebras only give $\sigma(K_4^{3-})\leq 0.09306286$. The minimum codegree threshold of $K_4^{3-}$ was determined to be $\pi_2(K_4^{3-})=1/4$ by Falgas-Ravry, Pikhurko, Vaughan and Volec~\cite{FalgasK4-} solving a conjecture by Nagle~\cite{codegreeconj}. The lower bound is obtained by a construction originally due to Erd\H{o}s and Hajnal~\cite{MR0337636}. Given a tournament $T$ on the vertex set $[n]$, define a $3$-graph $C(T)$ on $[n]$ by taking the edge set to be all triples of vertices inducing a cyclically oriented triangle in $T$. No tournament on $4$ vertices can contain more than two cyclically oriented triangles, hence $C(T)$ is $K_4^{3-}$-free. If $T$ is chosen uniformly at random, then the minimum codegree of $C(T)$ is $n/4-o(n)$ with high probability.

The uniform Tur\'an density of $K_{4}^{3-}$ was determined to be $\pi_u(K_4^{3-})=1/4$ by Glebov, Kr\'al' and Volec~\cite{MR3474967} using flag algebras. The same result was also obtained by Reiher, R\"odl, and Schacht~\cite{MR3790065} via regularity method of hypergraphs  without flag algebras. The lower bound is also achieved by $C(T)$.

\subsection{\texorpdfstring{\{$K_4^{3-},F_{3,2},C_5$\}}{TEXT}}
\label{K43minusF32C5}
Denote by $H_7$ the $3$-graph on seven vertices with edge set 
\begin{align*} E(H_7)=\{124,137,156,235,267,346,457,653,647, 621,542,517,431,327\}.
\end{align*}
Falgas-Ravry and Vaughan~\cite{RavryTuran} proved that $\sigma(K_4^{3-},F_{3,2},C_5)=12/49$, where the lower bound is achieved by the blow-up of $H_7$ on $n$ vertices.  However, the codegree squared density is achieved by a different $3$-graph, namely the complete balanced $3$-partite $3$-graph $S_n$. The $3$-graph $S_n$ is $K_4^{3-}$-free, $F_{3,2}$-free and $C_5$-free and it has codegree squared sum $\textup{co}_2(S_n)=1/27 n^4$. Thus, $\sigma(\{K_4^{3-},F_{3,2},C_5\})\geq 2/27$. Flag algebras give $\sigma(\{K_4^{3-},F_{3,2},C_5\})\leq 2/27$, hence,
\begin{align*}
\sigma(\{K_4^{3-},F_{3,2},C_5\})= \frac{2}{27}.
\end{align*}
For the minimum codegree threshold we have 
\begin{align*}
    \frac{1}{12}\leq\pi_2(\{K_4^{3-},F_{3,2},C_5\})\leq 0.186,
\end{align*}
where the upper bound is obtained by flag algebras and the lower bound by the following construction. Consider the intersection of the Erd\H{o}s-Hajnal~\cite{MR0337636} random tournament construction $C(T)$ with the $3$-graph having vertex partition $A_1\cup A_2 \cup A_3$, where $||A_i|-|A_j||\leq 1$ for $i,j\in [3]$, and edges $xyz$ for $x,y\in A_i,z\in A_{i+1}$ for $i=1,2,3$, where $A_4=A_1$. This $3$-graph has minimum codegree $n/12-o(n)$ with high probability, and is $C_5$, $F_{3,2}$ and $K_4^{3-}$-free.

For the uniform Tur\'an density we have
\begin{align*}
\pi_u(\{K_4^{3-},F_{3,2},C_5\})\leq \pi_u(F_{3,2})=0.
\end{align*}




\subsection{\texorpdfstring{\{$K_4^{3-},F_{3,2}$\}}{TEXT}}
\label{K43-F32}
Let $G$ be the blow-up of $S_6$ on $n$ vertices, recall $S_6$ was defined in Section~\ref{K43minus}. This $3$-graph $G$ is $K_4^{3-}$ and $F_{3,2}$-free. Falgas-Ravry and Vaughan~\cite{RavryTuran} proved that $\pi(K_4^{3-},F_{3,2})=5/18$ with the lower bound obtained by $G$. The codegree squared sum of $G$ is
\begin{align*} \textup{co}_2(G)= \binom{6}{2}\left(\frac{n}{6} \right)^2 \left( \frac{n}{3}\right)^2 = \frac{5}{108}n^4.
\end{align*}
Thus, $\sigma(\{K_4^{3-},F_{3,2}\})\geq 5/54$. Flag algebras give $\sigma(\{K_4^{3-},F_{3,2}\})\leq 5/54$, hence
\begin{align*}
\sigma(\{K_4^{3-},F_{3,2}\})=\frac{5}{54}.
\end{align*}
For the minimum codegree threshold we have 
\begin{align*}
    \frac{1}{12}\leq\pi_2(\{K_4^{3-},F_{3,2},C_5\})\leq \pi_2(\{K_4^{3-},F_{3,2}\} \leq 0.202,
\end{align*}
where the upper bound comes from flag algebras. For the uniform Tur\'an density we have
\begin{align*}
\pi_u(\{K_4^{3-},F_{3,2}\})\leq \pi_u(F_{3,2})=0.
\end{align*}


\subsection{\texorpdfstring{\{$K_4^{3-},C_5$\}}{TEXT}}
\label{K43-C5}
Denote by $G$ be the iterated blow-up of an edge on $n$ vertices. In other words, it is obtained from a complete balanced $3$-partite $3$-graph by inserting a complete balanced $3$-partite $3$-graph in each of the parts iteratively. 
The $3$-graph $G$ is $K_4^{3-} $ and $C_5$-free. This construction gives the current best lower bound on the Tur\'an density of $\{K_4^{3-},C_5\}$, see~\cite{RavryTuran}. For the codegree squared density, it gives $\sigma(\{K_4^{3-},C_5\})\geq {1}/{13}$.
We conjecture it to be the extremal example in $\ell_2$-norm.
\begin{conj}
\begin{align*}
\sigma(\{K_4^{3-},C_5\})=\frac{1}{13} \approx 0.0769230769.
\end{align*}
\end{conj}
Flag algebras give $\sigma(\{K_4^{3-},C_5\})\leq 0.07694083$ for the codegree squared density. For the minimum codegree threshold we have 
\begin{align*}
    \frac{1}{12}\leq\pi_2(\{K_4^{3-},F_{3,2},C_5\})\leq \pi_2(\{K_4^{3-},C_5\} \leq 0.202,
\end{align*}
where the upper bound comes from flag algebras. For the uniform Tur\'an density, we have
\begin{align*}
    \frac{2}{27}\leq \pi_u(\{K_4^{3-},C_5\}) \leq \pi_u(C_5) \leq \frac{4}{27},
\end{align*}
where the upper bound was obtained by Buci\'c, Cooper, Kr\'{a}l', Mohr and Munh\'a
Correia~\cite{personalComSamuel}. The lower bound is obtained by the following construction. Color pairs of vertices with colors red, blue, green uniformly at random, independently from each other. For a triple $u<v<w$ place an edge if $vw$ is green and one of $uv,uw$ blue and the other red.     

\subsection{\texorpdfstring{\{$F_{3,2},J_4$\}}{TEXT}}
\label{F32-J4}
Let $G$ be the blow-up of $K_4^3$ on $n$ vertices. This graph $G$ is  $F_{3,2}$-free and $J_4$-free. Falgas-Ravry and Vaughan~\cite{RavryTuran} proved that $\pi(\{F_{3,2},J_4\})=3/8$, where the lower bound is achieved by $G$. The $3$-graph $G$ has codegree squared sum $\textup{co}_2(G)= 3/32 n^4(1+o(1))$, thus, $\sigma(\{F_{3,2},J_4\})\geq 3/16$. Flag algebras give $\sigma(\{F_{3,2},J_4\})\leq 3/16$, hence
\begin{align*}
    \sigma(F_{3,2},J_4)=\frac{3}{16}.
\end{align*}
For the minimum codegree threshold we have 
\begin{align*}
    \frac{1}{12}\leq\pi_2(\{K_4^{3-},F_{3,2}\})\leq \pi_2(\{F_{3,2},J_4\} \leq \pi_2(F_{3,2})\leq \frac{1}{3},
\end{align*}
 because $K_4^{3-}$ is a subhypergraph of $J_4$. For the uniform Tur\'an density we have
\begin{align*}
\pi_u(\{F_{3,2},J_4\})\leq \pi_u(F_{3,2})=0.
\end{align*}

\subsection{\texorpdfstring{\{$F_{3,2},J_5$\}}{TEXT}}
\label{F32-J5}Falgas-Ravry and Vaughan~\cite{RavryTuran} proved that $\pi(\{F_{3,2},J_5\})=3/8$, where the blow-up of $K_4^3$ on $n$ vertices achieves the lower bound. Since $J_4\subset J_5$, we have 
\begin{align*}
    \frac{3}{16}\leq \sigma(\{F_{3,2},J_4\}) \leq \sigma(\{F_{3,2},J_5\}) \quad \quad \text{and} \quad \quad \frac{1}{12}\leq \pi_2(\{F_{3,2},J_4\}) \leq \pi_2(\{F_{3,2},J_5\}).
\end{align*}
By flag algebras we have for the codegree squared density $\sigma(\{F_{3,2},J_5\})\leq \frac{3}{16}$ and thus $\sigma(\{F_{3,2},J_5\})=3/16$, and for the minimum codegree threshold $\pi_2(\{F_{3,2},J_5\})\leq 0.28$. For the uniform Tur\'an density we have
\begin{align*}
\pi_u(\{F_{3,2},J_5\})\leq \pi_u(F_{3,2})=0.
\end{align*}

\subsection{\texorpdfstring{$J_4$}{TEXT}}
\label{sec:J4}
Let $G$ be the iterated blow-up of the complement of the Fano plane on $n$ vertices. This $3$-graph is $J_4$-free and conjectured to be an asymptotically sharp example in $\ell_1$-norm, see \cite{DaisyBollobas}. It has codegree squared sum of 
\begin{align*}
    \textup{co}_2(G)=\frac{48}{343} n^4 \sum \left(\frac{1}{7^4}\right)^i=\frac{7}{50} n^4.
\end{align*} 
Thus $\sigma(J_4)\geq 0.28$. We conjecture it to be the extremal example in $\ell_2$-norm. 
\begin{conj}
\begin{align*}
\sigma(J_4)=0.28.
\end{align*}
\end{conj}
Flag algebras only give $\sigma(J_4)\leq 0.28079696$. For the minimum codegree threshold and the uniform Tur\'an density of $J_4$ we have 
\begin{align*}
    \frac{1}{4} \leq \pi_2(K_4^{3-})\leq \pi_2(J_4)\leq 0.473 \quad \quad \quad \text{and} \quad \quad \quad \frac{1}{3}  \pi_u(J_4)\leq 4/9,
\end{align*}
where the upper bound on $\pi_2$ was obtained using flag algebras and the lower bound bound holds because $K_4^{3-} \subset J_4$. The uniform Tur\'an density bounds are from Theorem~5.6 in~\cite{MR3790065}.

\subsection{\texorpdfstring{$J_5$}{TEXT}}
\label{sec:J5}
Since $J_4\subset J_5$, we get the following lower bounds:
\begin{align*}
    \frac{1}{2} \leq \pi(J_4)\leq \pi(J_5), \   0.28 \leq \sigma(J_4)\leq \sigma(J_5), \quad \text{and} \quad  \frac{1}{4} \leq \pi_2(J_4)\leq \pi_2(J_5).
\end{align*}
Flag algebras give us the following upper bounds
\begin{align*}
     \pi(J_5)\leq 0.64474184, \quad \quad  \sigma(J_5)\leq  0.4427457, \quad \quad  \pi_2(J_5)\leq 0.613. \quad \quad  
\end{align*}
It would be interesting, if one could separate the parameters of $J_4$ and $J_5$ from each other. For the uniform Tur\'an density, Reiher, R\"odl and Schacht~\cite{MR3790065} proved $7/16 \leq \pi_u(J_5)\leq 9/16$.

\subsection{\texorpdfstring{$C_5$}{TEXT}}
\label{sec:C5}
Denote by $G(n)$ the $3$-graph on $n$ vertices where the vertex set is partitioned into two sets $A$ and $B$ of sizes $bn$ and $(1-b)n$, where $b\in(0,1)$ and $E(G(n))$ consists of all triples $xyz$, where $x,y\in B$ and $z\in A$. This $3$-graph is $C_5$-free. Let $G_{it}$ be the $3$-graph constructed from $G(n)$ by iteratively adding $G(bn)$ inside the class $B$. The $3$-graph $G_{it}$ is also $C_5$-free. This construction gives the current best-known lower bound on the Tur\'an density of $C_5$, see~\cite{F32Mubayi}. It's codegree squared sum can be lower bounded as follows
\begin{align*}
    \textup{co}_2(G_{it}) > \left(b^2\frac{(1-b)^2}{2} +  b(1-b)b^2 \right)  \left( \frac{1}{1-(1-b)^4}\right)n^4 > 0.12597 n^4
\end{align*}
for $b$ of size about $0.7$, thus $\sigma(C_5)\geq 0.25194$. Flag algebras give $\sigma(C_5)\leq 0.25310457$. For the minimum codegree threshold of $C_5$ we have
\begin{align*}
    \frac{1}{3}\leq \pi_2(C_5)\leq 0.3993,
\end{align*}
where the upper bound comes from flag algebras and the lower bound from the graph with vertex partition $A_1\cup A_2 \cup A_3$, where $||A_i|-|A_j||\leq 1$ for $i,j\in [3]$, and edges $xyz$ for $x,y\in A_i,z\in A_{i+1}$ for $i=1,2,3$, where $A_4=A_1$.  

The uniform Tur\'an density of $C_5$ is $\sigma(C_5)=4/27$. The lower bound construction was presented by Reiher in \cite{MR4111729} and the upper bound proof by Buci\'c, Cooper, Kr\'{a}l', Mohr and Munh\'a
Correia~\cite{personalComSamuel}. We note that a standard application of the method of flag algebras only gives an upper bound of $\sigma(C_5)\leq 0.402$.

\subsection{\texorpdfstring{$C_5^-$}{TEXT}}
\label{sec:C5minus}
Denote by $G$ be the iterated blow-up of an edge on $n$ vertices. The $3$-graph $G$ is $C_5^-$-free. Very recently, Lidick\'y, Mattes and Pfender~\cite{BernardFlo}, and independently Bodn\'ar, Le\'on, Liu and Pikhurko~\cite{OlegXizhi} proved that this is the asymptotical extremal example for the Tur\'an density, i.e. $\pi(C_5^-)=1/4$. For the codegree squared density, it gives $\sigma(C_5^-)\geq {1}/{13}$.
We conjecture it to be the extremal example in $\ell_2$-norm.
\begin{conj}
\begin{align*}
\sigma(C_5^-)=\frac{1}{13} \approx 0.0769230769.
\end{align*}
\end{conj}
Flag algebras give $\sigma(C_5^-)\leq 0.07725405$ for the codegree squared density and $\pi_2(C_5^-)\leq 0.136$ for the minimum codegree threshold. Recently, Piga, Sales and Sch\"ulke~\cite{PiSaSc} proved that $\pi_2(C_5^-)=0$.

The uniform Tur\'an density of $C_5^-$ is $\pi_u(C_5^-)=0$. This was observed by Reiher, R\"{o}dl and Schacht~\cite{MR3764068} using Theorem~\ref{uniformturan}. For the sake of completeness, we repeat the proof here. Let $v_1,v_2,v_3,v_4,v_5$ form a $C_5^-$, that is $v_1v_2v_3,v_2v_3v_4,v_3v_4v_5,v_4v_5v_1$ are edges. Consider the ordering $v_3<v_2<v_1<v_5<v_4$ of the vertices. The edges of the shadow graph can be colored in the following way. Color the edges $v_3v_2,v_3v_5,v_1v_5$ red, the edges $v_1v_3,v_3v_4,v_1v_4$ blue, and the edges $v_1v_2,v_2v_4,v_4v_5$ green; see Figure~\ref{fig:colorings}. Thus, property (b) of Theorem~\ref{uniformturan} holds.

\subsection{\texorpdfstring{$K_5^=$}{TEXT}}
\label{sec:K_5^=}
Let $K_5^=$ be the $3$-graph on $5$ vertices with $8$ edges, where the two missing edges intersect in exactly one vertex. We have
\begin{align*}
   0.58656 \leq \frac{227}{387}\leq \pi(K_5^=)\leq 0.60961944 \quad \text{and} \quad 0.357942 \leq \frac{2909}{8127}\leq \sigma(K_5^=)\leq 0.388725.
    \end{align*}
The hypergraph providing both lower bounds is the following. Take the $3$-uniform hypergraph on $n$ vertices with vertex set $V_1\cup V_2 \cup V_3$ such that $||V_i|-|V_j||\leq 1$ for $i\neq j$ and edge set 
\begin{align*}
    \{abc: a\in V_1,b\in V_2, c\in V_3 \} \cup \{abc: a,b\in V_1,c\in V_2 \} \\ \cup \ \{abc: a,b\in V_2,c\in V_3 \} \cup  \{abc: a,b\in V_3,c\in V_1 \}.  
\end{align*} Now, place in each of the three classes an iterated blowup of $S_6$, recall $S_6$ was defined in Section~\ref{K43minus}. This $3$-graph has
\begin{align*}
    \frac{5}{54}n^3+30  \sum\left( \frac{n}{3} \cdot \frac{1}{6^i}\right)^3 +o(n^3)=\frac{227}{2322}n^3(1+o(1))\geq 0.58656 \binom{n}{3}(1+o(1))
\end{align*}
edges, codegree squared sum of
\begin{align*} \frac{4}{27}n^4 + 3 \frac{1}{2} \left(\frac{n}{3}\right)^4 \sum_{k\geq 1} \frac{5}{6^k} \left(\frac{2}{6^k}+1 \right)^2= \frac{2909}{8127} \frac{n^4}{2},
\end{align*}
and is $K_5^=$-free. We conjecture the lower bounds to be asymptotically sharp. 
\begin{conj}
 \begin{align*}\pi(K_5^=)=\frac{227}{387} \quad \quad \text{and} \quad \quad \sigma(K_5^=)=\frac{2909}{8127}.
 \end{align*}
\end{conj}
For the minimum codegree threshold and uniform Tur\'an density of $K_5^=$, we have 
\begin{align*}
    \frac{1}{2} \leq \pi(K_4^3)\leq \pi_2(K_5^=)\leq 0.569 \quad \quad \text{and} \quad \quad   \frac{1}{2} \leq \pi_u(K_4^3)\leq \pi_u(K_5^=)\leq 0.567,
\end{align*}
where the upper bound is obtained via flag algebras. It would be interesting to decide if those parameters are equal to each other. 

\subsection{\texorpdfstring{$K_5^<$}{TEXT}}
\label{sec:K_5^<}
Let $K_5^<$ be the $3$-graph on $5$ vertices with $8$ edges, where the two missing edges intersect in exactly two vertices. Since $K_4^3$ is a subgraph of $K_5^<$, we have 
\begin{gather*}
\frac{5}{9} \leq \pi(K_4^3) \leq
    \pi(K_5^<), \quad  \frac{1}{3} \leq \sigma(K_4^3 ) \leq \sigma(K_5^<), \quad
    \frac{1}{2} \leq \pi_2(K_4^3) \leq \pi_2(K_5^<) \\
    \text{and} \quad \quad \frac{1}{2} \leq \pi_u(K_4^3) \leq \pi_u(K_5^<).
\end{gather*}
Using flag algebras we obtain 
\begin{gather*}
\pi(K_5^<)\leq 0.57704775,\quad \sigma(K_5^<)\leq 0.3402107, \quad \pi_2(K_5^<)\leq 0.560 \quad \text{and} \quad \pi_u(K_5^<)\leq 0.568.
\end{gather*}
We conjecture the lower bounds to be sharp.
\begin{conj}
\begin{align*}
\pi(K_5^<)=\frac{5}{9}, \quad \quad \sigma(K_5^<)=\frac{1}{3}, \quad \quad \pi_2(K_5^<)=\frac{1}{2} \quad \quad \text{and} \quad \quad \pi_u(K_5^<)=\frac{1}{2}.    
\end{align*}
\end{conj}

\subsection{\texorpdfstring{$K_5^-$}{TEXT}}
\label{sec:K_5^-}
Let $K_5^-$ be the $3$-graph on $5$ vertices with $9$ edges. Markstr\"om~\cite{MR3190267} presented a construction giving $\pi(K_5^-)\geq \frac{46}{81}$. Note that $\pi(K_5^-) \geq \pi(K_5^=)\geq \frac{227}{387}$ giving a slight improvement of this. Very recently, after the appearance of this article, Liu, Sch\"ulke, Wang, Yang and Zhang~\cite{HongBjarne} presented an improved construction giving $\pi(K_5^-) \geq 0.602673$. We have 
\begin{gather*}
 \frac{2909}{8127} \leq \sigma(K_5^=) \leq \sigma(K_5^-), \quad \frac{1}{2} \leq\pi_2(K_4^3) \leq \pi_2(K_5^-) \quad
\text{and}  \quad \frac{1}{2}\leq \pi_u(K_4^3)\leq \pi_u(K_5^-).
\end{gather*}
It would be interesting to decide if those parameters are equal to each other. Using flag algebras we obtain 
\begin{align*}
    \pi(K_5^-)\leq 0.64208589,\quad \sigma(K_5^-)\leq 0.41961743,\quad \pi_2(K_5^-)\leq 0.621 \quad \text{and} \quad \pi_u(K_5^-)\leq 0.626.
\end{align*}

\section{Proof of Theorem~\ref{F5exact}}
\label{F5section}
Bollob\'as~\cite{Bollobascancellative} proved that the $\{F_4,F_5\}$-free $3$-graph with the largest number of edges is $S_n$, the complete balanced $3$-partite $3$-graph. His idea was to construct an almost partition of three pairwise disjoint sets with the property that every edge has at most one point from each and then use a counting argument to show that the number of edges is maximized when the classes form a partition and have equal size. 
This result of Bollob\'as~\cite{Bollobascancellative}  was extended by Frankl and F\"uredi~\cite{F5Frankl} who proved that for $n\geq 3000$ the $F_5$-free $3$-graph with the largest number of edges is $S_n$. Keevash and Mubayi~\cite{MubayiF5} proved a stability result for $F_5$-free graphs in the $\ell_1$-norm using a different method from the one in \cite{Bollobascancellative}. There is a straightforward modification of Bollob\'as' result which we will use to get a stability result for $F_5$-free graphs in the $\ell_2$-norm.

\begin{theo}
\label{F5stability}
For every $\varepsilon > 0$ there exists $\delta>0$ and $n_0\in \mathbb{N}$ such that if $G$ is an $F_5$-free $3$-uniform hypergraph on $n \geq n_0$ vertices with $\textup{co}_2(G)\geq (1-\delta) \frac{1}{27} n^4$, then we can partition $V(H)=A \cup B \cup C$ such that 
\begin{align}
\label{stabcondition1} 
  e(A)+e(B)+e(C)+e(A,B)+e(B,C)+e(A,C)\leq \varepsilon n^3
\end{align}
and 
\begin{align}
\label{stabcondition2} 
e(A,B,C)\geq \frac{1}{27}n^3 -\varepsilon n^3.
\end{align}
\end{theo}
Note that, unlike the $\ell_1$-norm case, here it is not instantly obvious why \eqref{stabcondition1} and \eqref{stabcondition2} are the same conditions, as having high codegree square sum does not imply the existence of many edges.
\begin{proof}
Let $\varepsilon>0$ and choose a constant $\delta=\delta(\varepsilon)$ such that the following argument holds. Let $G$ be an $n$-vertex $F_5$-free $3$-graph satisfying $\textup{co}_2(G)\geq (1-\delta)1/27n^4$. Choose $v_1,v_2\in V(G)$ such that $d(v_1,v_2)$ is maximum among all pairs of vertices. Denote $A'=N(v_1,v_2)$ and note that there is no hyperedge with two vertices from $A'$ and one vertex from $V(G)\setminus\{v_1,v_2\}$ because $G$ is $F_5$-free, see Figure~\ref{fig:ABC} for an illustration.
Choose $v_3\in A', v_4\in V(G)\setminus (A' \cup \{v_1,v_2\})$ such that $d(v_3,v_4)$ is maximum and denote $B'=N(v_3,v_4)\setminus \{v_1,v_2\}$. We have $A'\cap B'=\emptyset$, because for $x\in A'\cap B'$, $v_3v_4x$ is an edge with two vertices in $A'$ and one in $V(G)\setminus\{v_1,v_2\}$. Again, there are no edges with two vertices from $B'$ and one vertex from $V(G)\setminus\{v_1,v_2,v_3,v_4\}$ because $G$ is $F_5$-free. Next, choose $v_5\in A'\setminus \{v_3\},\ v_6\in B' \setminus \{v_1,v_2,v_3,v_4\}$ such that $d(v_5,v_6)$ is maximum. Set $C'=N(v_5,v_6)\setminus\{v_1,v_2,v_3,v_4\}$. We have $C' \cap A'=\emptyset$ and $C'\cap B'=\emptyset$ as otherwise there is an edge with two vertices in either $A'$ or $B'$ and one in $V(G)\setminus \{v_1,v_2,v_3,v_4\}$, which would give a copy of $F_5$. Set $E:=\{v_1,v_2,v_3,v_4,v_5,v_6\}$. There is no edge with two vertices in $C'$ and one in $V(G)\setminus E$, because if $c_1c_2x$ is such an edge, then $\{v_5,v_6,c_1,c_2,x\}$ induces an $F_5$. Set
\begin{align*}
A:= A' \setminus E, \quad    B:= B' \setminus E,  \quad D=V(G) \setminus (A\cup B\cup C'), \quad  C=C' \cup D=V(G)\setminus(A \cup B)
\end{align*}
and note that $V(G)=A\cup B\cup C' \cup D=A\cup B\cup C$. Denote $a=|A|/n$, $b=|B|/n$, $c'=|C'|/n$ and $d=|D|/n$. Thus, $a+b+c'+d=1$.

\begin{figure}
\begin{center}
    \begin{tikzpicture}
\draw 
 (-1,-0.2) coordinate (1) node[vtx,label=left:$v_1$]{}
 (-1,1.2) coordinate (2) node[vtx,label=left:$v_2$]{}
 (0.3,0.5) coordinate (A) node[vtx]{}
 (1,0.5) circle(1cm) node {$A'$} 
 (4,0.5) circle(1cm) node {$B'$} 
 (2.5,-1.5) circle(1cm)  node {$C'$}
 (1.5,1) coordinate (3) node[vtx,label=left:$v_3$]{}
 (2.5,2.5) coordinate (4) node[vtx,label=above:$v_4$]{}
 (3.5,1) coordinate (B) node[vtx,label=right:$ $]{}
 (1.5,0) coordinate (5) node[vtx,label=left:$v_5$]{}
 (3.5,0) coordinate (6) node[vtx,label=right:$v_6$]{}
 (2.5,-1.2) coordinate (C) node[vtx,label=right:$ $]{}
 ;
\draw[hyperedge] (1) to[bend left] (A) to[bend left] (2) to[bend left] (1);
\draw[hyperedge] (3) to[bend left] (B) to[bend left] (4) to[bend left] (3);
\draw[hyperedge] (5) to[bend left] (C) to[bend left] (6) to[bend left] (5);
\end{tikzpicture}
\end{center}
    \caption{$A',B',C'$ and $v_1,v_2,v_3,v_4,v_5,v_6$ from the proof of Theorem~\ref{F5stability}.}
    \label{fig:ABC}
\end{figure}
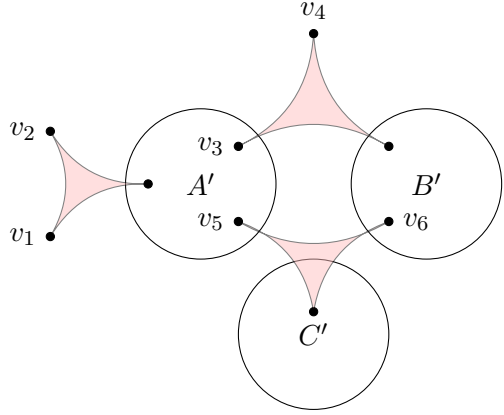

\begin{claim}
\label{classsizesofblobs}
We have 
\begin{align*}
    \left\vert a-\frac{1}{3}\right\vert\leq \frac{\varepsilon}{10} ,\quad   \quad \left\vert b-\frac{1}{3}\right\vert\leq \frac{\varepsilon}{10} ,\quad \quad \left\vert c'-\frac{1}{3}\right\vert\leq \frac{\varepsilon}{10} ,\quad \quad d\leq \frac{\varepsilon}{10}. 
\end{align*}
\end{claim}
\begin{proof}
Assume, toward a contradiction, that the claim is false.
The only triples with two vertices from either $A$, $B$ or $C'$ which can be edges are those with one vertex in $E$. Therefore, we get
\begin{align*}
   d(x,y)\leq 6 \text{ for } x,y\in A, \quad
   d(x,y)\leq 6\text{ for } x,y\in B \quad \text{and} \quad 
   d(x,y)\leq 6 \quad \text{for} \quad x,y\in C'.
\end{align*}
By the choices of the vertices in $E$, we further have 
\begin{equation*}
  d(x,y)\leq c'n+6 \quad \text{for} \quad x\in A,\  y\in B, \quad d(x,y)\leq bn+6 \quad \text{for} \quad x\in A,\ y\in C' \cup D \setminus E
\end{equation*}
\begin{equation*}
    \text{and} \quad d(x,y)\leq an \quad \text{for} \quad x,y \in V(G).
\end{equation*}
Now,
\begin{align}
\nonumber
\label{opti}
 \frac{1}{27}(1-\delta)n^4 &\leq \textup{co}_2(G) \leq ab(c'n+6)^2n^2+an(c'n+dn+6)(bn+6)^2 + |E|n(an)^2\\
&+\left( \binom{(b+c'+d)n}{2}-\binom{bn}{2}-\binom{c'n}{2} \right)(an)^2 \nonumber\\
&\leq \left( abc'^2+a(c'+d)b^2 + \left( \frac{(b+c'+d)^2}{2}-\frac{b^2}{2}-\frac{c'^2}{2} \right)a^2 \right)n^4+100n^3 \nonumber \\
&= f(a,b,c',d)n^4 +100n^3, 
\end{align}
where $f(w,x,y,z)=wxy^2+w(y+z)x^2 +\left( \frac{(x+y+z)^2}{2}-\frac{x^2}{2}-\frac{y^2}{2} \right)w^2$. Define
\begin{align*}
    R=\{(w,x,y,z) \ : \ w,x,y,z \geq 0, \ w+x+y+z=1 \} \subset \mathbb{R}^4.
\end{align*}
The polynomial $f$ has a maximum on $R$ of $1/27$ attained at unique point $w,x,y=1/3, \ z=0$.
This can be checked with a tedious calculation or using a computer, we omit the details. Define the open ball
\begin{align*}
    U:=\left\{(w,x,y,z) \ : \ \left\vert w-\frac{1}{3}\right\vert < \frac{\varepsilon}{10} ,\quad   \left\vert x-\frac{1}{3}\right\vert <\frac{\varepsilon}{10} ,\quad  \left\vert y-\frac{1}{3}\right\vert< \frac{\varepsilon}{10} ,\quad  z < \frac{\varepsilon}{10} \right\}.
\end{align*}
The regions $R$ and  $R\setminus U$ are closed and bounded, thus also compact on $\mathbb{R}^4$. Therefore, the polynomial $f$ attains a maximum on $R\setminus U$. Set 
\begin{align*}
    \delta:= \min \left\{\varepsilon,\frac{1}{27}- \max_{(w,x,y,z)\in R \setminus U} f(w,x,y,z) \right\},
\end{align*}
and note that $\delta=\delta(\varepsilon)>0$, since the unique maximum of $f$ on $R$ is also in $U$. The point $(a,b,c',d)\in R\setminus U$ by assumption and therefore $f(a,b,c',d)\leq 1/27-\delta$. Combining this observation with inequality \eqref{opti}, we get
\begin{align*}
    \frac{1}{27}(1-\delta)n^4 \leq  f(a,b,c',d)n^4 +100n^3 \leq \frac{1}{27}n^4-\delta n^4 +100n^3 < \frac{1}{27}(1-\delta)n^4,
\end{align*}
for $n$ large enough. This is a contradiction and thus Claim~\ref{classsizesofblobs} holds.
\end{proof}

Denote $H=G[A\cup B \cup C']$. Then, 
\begin{align*}
\textup{co}_2(H)\geq \textup{co}_2(G)-\frac{6}{10}\varepsilon n^4 \geq (1-\delta)\frac{1}{27}n^4 -\frac{6}{10}\varepsilon n^4 \geq \frac{1}{27}n^4-\varepsilon n^4, \end{align*}
because the number of edges incident to $D$ is at most $(\varepsilon/10) n^3$ by Claim~\ref{classsizesofblobs} and the codegree squared sum of a $3$-graph decreases by at most $6n$ when one edge is removed. The $3$-graph $H$ is $3$-partite. Thus, for each edge $e\in E(H)$, $w_{H}(e)\leq n$. By double-counting, we get a lower bound on the size of $E(H)$,
\begin{align*}
     \frac{1}{27}n^4-\varepsilon n^4 \leq \textup{co}_2(H) = \sum_{e\in E(H)} w_{H}(e)\leq |E(H)|n.
\end{align*}
Consider the vertex partition $V(G)=A \cup B \cup C$. For the number of cross-edges, we have 
\begin{align*}
    e(A,B,C) \geq e(A,B,C')=|E(H)|\geq \frac{1}{27}n^3-\varepsilon n^3 
\end{align*}
and thus the partition $V(G)=A \cup B \cup C$ satisfies \eqref{stabcondition2}.
Since there is no edge in $G$ with two vertices in one of the classes $A,B$ or $C'$ and because there are at most $(\varepsilon/10) n^3+6n^2$ edges incident to $D\cup E$ by Claim~\ref{classsizesofblobs}, the partition $V(G)=A\cup B \cup C$ also satisfies 
\begin{align*}
    &  e(A)+e(B)+e(C)+e(A,B)+e(B,C)+e(A,C) \\
    &=e(C)+e(B,C)+e(A,C)\leq \frac{\varepsilon}{10} n^3+6n^2 \leq \varepsilon n^3.
\end{align*}
Hence, \eqref{stabcondition1} holds, completing the proof of Theorem~\ref{F5stability}.  
\end{proof}
Now, we will prove the exact result for $F_5$-free $3$-graphs under an additional universal minimum-degree-type assumption. Afterwards we will deal with the additional assumption.

\begin{theo}
\label{F5mindegasumption}
There exists $n_0\in \mathbb{N}$ such that for all $n\geq n_0$ the following holds. If $G$ is an $F_5$-free $3$-graph satisfying  
\begin{align*}
\textup{co}_2(G)\geq \textup{co}_2(S_n)
\end{align*}
and 
\begin{align}
\label{minqassumption2}
    q(x):= \sum_{y\in V,y\neq x} d(x,y)^2 + 2\sum_{vw\in E(L(x))}d(x,y) \geq \frac{4}{27}n^3-10n^2=:d(n)
\end{align}
for all $x\in V(G)$, then $G \cong S_n$.
\end{theo}
\begin{proof}
Let $G$ be an $F_5$-free $3$-graph on $n\geq n_0$ vertices satisfying $\textup{co}_2(G)\geq \textup{co}_2(S_n)$.
Choose $\varepsilon=10^{-20}$. Apply Theorem~\ref{F5stability} to $G$ and get a partition $A\cup B \cup C$ of the vertex set such that $e(A,B,C)\geq \frac{1}{27}n^3-\varepsilon n^3$. Among all such partitions we choose one which maximizes $e(A,B,C)$. We start by making an observation about the class sizes.
\begin{claim} We have
\label{classsizes1}
\begin{align*}\frac{n}{3} - \varepsilon^{1/4} n \leq |A|,|B|,|C| \leq \frac{n}{3} + \varepsilon^{1/4} n.
\end{align*}
\end{claim}
\begin{proof}
Without loss of generality, let $|A|\leq |B|\leq |C|$. Assume for contradiction that Claim~\ref{classsizes1} is false, so $|A|< n/3 - \frac{\varepsilon^{1/4}}{2} n$. We have 
\begin{align}
    e(A,B,C) &\leq |A||B||C|\leq |A| \left(\frac{n-|A|}{2}\right)^2 =
    (n-(n-|A|)) \left(\frac{n-|A|}{2}\right)^2 \label{co2setsizes} \nonumber \\ 
    \nonumber &\leq
    \left( \frac{n}{3}-\frac{\varepsilon^{1/4}}{2} n\right)\left(\frac{\frac{2n}{3}+\frac{\varepsilon^{1/4}}{2} n}{2}\right)^2
    =     \left( \frac{n}{3}-\frac{\varepsilon^{1/4}}{2} n\right)\left(\frac{n}{3}+\frac{\varepsilon^{1/4}}{4}n\right)^2 \\
    &= \left( \frac{1}{27}-\frac{\varepsilon^{3/4}}{32}- \frac{\varepsilon^{1/2}}{16}  \right)n^3 < \frac{1}{27}n^3-\varepsilon n^3,
\end{align} 
where the second-to-last inequality holds, because the function $x \rightarrow (n-x)(x/2)^2$ is monotone decreasing for $x>2n/3$. Inequality \eqref{co2setsizes} is in contradiction with $e(A,B,C)\geq \frac{1}{27}n^3-\varepsilon n^3$. Hence, the claim is true.  
\end{proof}
Define junk sets $J_A,J_B,J_C$ to be the sets of vertices which are not typical. To be precise,  
\begin{align*}
J_A:= \{x\in A: \ |L_{B,C}(x)|\leq |B||C|-\sqrt{\varepsilon}n^2 \}, \\
J_B:= \{x\in B: \  |L_{A,C}(x)|\leq |A||C|-\sqrt{\varepsilon}n^2 \}, \\
J_C:= \{x\in C: \  |L_{A,B}(x)|\leq |A||B|-\sqrt{\varepsilon}n^2 \},
\end{align*}
where $|G| = |E(G)|$. 
We have $|J_A|,|J_B|,|J_C|\leq \sqrt{\varepsilon}n$ as otherwise $e(A,B,C)\leq |A||B||C|-\varepsilon n^3$.

\begin{claim} 
\label{cleaning1}There is no edge $xyz$ with $z\in V(G)$, $x,y\in A \setminus J_A$ or $x,y\in B \setminus J_B$ or $x,y\in C \setminus J_C$.
\end{claim}
\begin{proof}
Let $xyz$ be an edge with $z\in V(G)$, $x,y\in A \setminus J_A$. We have $|L_{B,C}(x)|,|L_{B,C}(y)|\geq |B||C|-\sqrt{\varepsilon}n^2$ and thus $E(L_{B,C}(x)) \cap E(L_{B,C}(y)) \neq \emptyset$. Let $\{b,c\}\in E(L_{B,C}(x)) \cap E(L_{B,C}(y))$ with $b,c\neq z$. Now, $\{x,y,z,b,c\}$ induces an $F_5$, a contradiction. The remaining parts of the statement follows similarly. 
\end{proof}

\begin{claim}
\label{cleaning2}
For $v\in V(G)$, we have 
\begin{align*}
    |L_A(v)|\leq \sqrt{\varepsilon}n^2,\quad \quad |L_B(v)|\leq \sqrt{\varepsilon}n^2, \quad \quad \text{and} \quad \quad |L_C(v)| \leq \sqrt{\varepsilon}n^2.
\end{align*}
Further, 

\begin{itemize}
    \item for $a\in A\setminus J_A$, we have $|L_{A,B}(a)|,|L_{A,C}(a)| \leq \sqrt{\varepsilon}n^2$, 
    \item for $b\in B\setminus J_B$, we have 
    $|L_{A,B}(b)|,|L_{B,C}(b)| \leq \sqrt{\varepsilon}n^2$, 
    \item for $c\in C\setminus J_C$, we have 
    $|L_{B,C}(c)|,|L_{A,C}(c)| \leq \sqrt{\varepsilon}n^2$.
\end{itemize}
\end{claim}
\begin{proof}
Let $v\in V(G)$. We have $|L_A(v)|\leq |J_A||A|\leq  \sqrt{\varepsilon}n^2$, because by Claim~\ref{cleaning1} every edge in the link graph $L_A(v)$ needs to be incident to a vertex in $J_A$. Similarly, we get $|L_B(v)|,|L_C(v)|\leq \sqrt{\varepsilon}n^2$. \\
Now, let $a\in A\setminus J_A$. We have $|L_{A,B}(a)| \leq |J_A||B|\leq  \sqrt{\varepsilon}n^2$, because by Claim~\ref{cleaning1} every edge in the link graph $L_{A,B}(a)$ needs to be incident to a vertex in $J_A$. Similarly, we get $|L_{A,C}(a)| \leq \sqrt{\varepsilon}n^2$. Further, the other two statements follow by the same reasoning.

\end{proof}

\begin{claim} 
\label{cleaning3} Let $xyz$ be an edge with 
\begin{align*}
    &x\in A\setminus J_A, \ y\in A, \ z\in (B\setminus J_B) \cup (C\setminus J_C) \text{ or } \\[0.5em]
    &x\in C\setminus J_C, \ y\in C, \ z\in (A\setminus J_A) \cup (B\setminus J_B) \text{ or } \\[0.5em]
    &x\in B\setminus J_B, \ y\in B, \ z\in (A\setminus J_A) \cup (C \setminus J_C). 
\end{align*}
Then, $|L_{A,B}(y)|,|L_{A,C}(y)|, |L_{B,C}(y)| \leq 2\sqrt{\varepsilon}n^2$.
\end{claim}
\begin{proof}
Let $xyz$ be an edge with $x\in (A\setminus J_A), y\in A, z\in (B\setminus J_B) \cup (C\setminus J_C)$. Without loss of generality, let $z\in B\setminus J_B$. Since $|L_{B,C}(x)|\geq |B||C|-\sqrt{\varepsilon}n^2$, we have $|L_{B,C}(y)|\leq 2\sqrt{\varepsilon}n^2$, as otherwise $E(L_{B,C}(x)) \cap E(L_{B,C}(y))\neq \emptyset$, allowing us to find an $F_5$. Also, since $|L_{A,C}(z)|\geq |A||C|-\sqrt{\varepsilon}n^2$, we have $|L_{A,C}(y)|\leq 2\sqrt{\varepsilon}n^2$ by the same reasoning.  Since we chose the partition $A\cup B \cup C$ such that $e(A,B,C)$ is maximized, we have $|L_{A,B}(y)|\leq |L_{B,C}(y)| \leq 2\sqrt{\varepsilon}n^2$ as otherwise moving $y$ to $C$ increases the number of those edges having one vertex in each set $A,B,C$. The other two statements follow by a similar argument. 
\end{proof}

\begin{claim}
\label{cleaning4}
There is no edge $xyz$  with 
\begin{align*}
    &x\in A\setminus J_A, \ y\in A, \ z\in (B\setminus J_B) \cup (C\setminus J_C) \text{ or } \\[0.5em]
    &x\in C\setminus J_C, \ y\in C, \ z\in (A\setminus J_A) \cup (B\setminus J_B) \text{ or } \\[0.5em]
    &x\in B\setminus J_B, \ y\in B, \ z\in (A\setminus J_A) \cup (C \setminus J_C). 
\end{align*}
\end{claim}
\begin{proof}
Let $xyz\in E(G)$ be an edge with $x\in (A\setminus J_A), y\in A, z\in (B\setminus J_B) \cup (C\setminus J_C)$.  By Claim~\ref{cleaning2}, $|L_A(y)|,|L_B(y)|,|L_C(y)|\leq \sqrt{\varepsilon}n^2$ and by Claim~\ref{cleaning3}, $|L_{A,B}(y)|,|L_{A,C}(y)|, |L_{B,C}(y)| \leq 2\sqrt{\varepsilon}n^2$. Therefore, $|L(y)|\leq 9 \sqrt{\varepsilon}n^2$. Our strategy for proving this claim is to upper bound $q(y)$ violating \eqref{minqassumption2}. We have 
\begin{align}
\label{mindegreecheck3}
    2\sum_{\{b,c\}\in E(L(y))}d(b,c) \leq 2\sum_{\{b,c\}\in E(L(y))}n \leq  18 \sqrt{\varepsilon}n^3\leq \frac{1}{100}n^3.
\end{align}
The number of vertices $v\in V(G)$ satisfying $d(y,v)\geq 9\varepsilon^{1/4}n$ is at most $2\varepsilon^{1/4} n$ as otherwise 
$|L(y)|>9 \sqrt{\varepsilon}n^2$.
Using this fact, we get 
\begin{align}
\label{mindegreecheck4}
    \sum_{v\in V(G)}d(y,v)^2 \leq 2\varepsilon^{1/4} n^3 + n (9\varepsilon^{1/4}n)^2\leq \frac{1}{100} n^3. 
\end{align}
Combining \eqref{mindegreecheck3} and \eqref{mindegreecheck4}, we upper bound
\begin{align*}
    q(y) &\leq  \sum_{v\in V(G)}d(y,v)^2 + 2\sum_{\{b,c\}\in E(L(y))}d(b,c) \leq  \frac{2}{100} n^3 < \frac{4}{27}n^3 - 10n^2 ,
\end{align*}
violating \eqref{minqassumption2}. Hence, there cannot exist an edge $xyz$ with $x\in (A\setminus J_A), y\in A, z\in (B\setminus J_B) \cup (C\setminus J_C)$. The remaining two statements of this claim follow by a similar argument.
\end{proof}

\begin{claim} 
\label{cleaning5}
The following three statements hold.
\begin{itemize}
    \item For $a\in A$ we have $|L_{A,B}(a)|,|L_{A,C}(a)|\leq 2\sqrt{\varepsilon}n^2$.
    \item For $b\in B$ we have $|L_{A,B}(b)|,|L_{B,C}(b)|\leq 2\sqrt{\varepsilon}n^2$.
    \item For $c\in C$ we have $|L_{A,C}(c)|,|L_{B,C}(c)|\leq 2\sqrt{\varepsilon}n^2$. 
\end{itemize}
\end{claim}
\begin{proof}
Let $a\in A$. We have $|L_{A,B}(a)|,|L_{A,C}(a)|\leq 2\sqrt{\varepsilon}n^2$, because by Claim~\ref{cleaning4} every edge in both link graphs need to have one vertex in a junk set. The other two statements follow by a similar argument. 
\end{proof}

\begin{claim} 
\label{cleaning6}
There is no edge $x_1x_2v$ with $v\in V(G)$ and $x_1,x_2\in A$ or $x_1,x_2\in B$ or $x_1,x_2\in C$. 
\end{claim}
\begin{proof}
Without loss of generality, let $x_1x_2v$ be an edge where $x_1,x_2\in A, \ v\in V(G)$. Then,  $|L_{B,C}(x_1)| \leq \frac{2}{3} |B||C|$ or $|L_{B,C}(x_2)| \leq \frac{2}{3} |B||C|$, as otherwise $E(L_{B,C}(x_1)) \cap E(L_{B,C}(x_2))\neq \emptyset$ and thus there are two vertices $b\in B,c\in C,b,c\neq v$ such that $\{x_1,x_2,v,b,c\}$ induces an $F_5$ in $G$. Without loss of generality let $|L_{B,C}(x_1)| \leq \frac{2}{3} |B||C|$. By Claim~\ref{cleaning2} and Claim~\ref{cleaning5}, we have
\begin{align*}
   |L_A(x_1)|,|L_B(x_1)|,|L_C(x_1)|\leq \sqrt{\varepsilon}n^2 \quad \text{and} \quad |L_{A,B}(x_1)|,|L_{A,C}(x_1)|\leq 2\sqrt{\varepsilon}n^2. 
\end{align*}
Our strategy for proving this claim is to upper bound $q(x_1)$ again violating  \eqref{minqassumption2}. For all $b\in B\setminus J_B,c\in C\setminus J_C$ we have $d(b,c)\leq |A|\leq n/3 + \varepsilon^{1/4}n$ by Claim~\ref{cleaning4} and Claim~\ref{classsizes1}. Thus, $d(b,c)\leq n/3 + \varepsilon^{1/4}n$ for all but at most $\sqrt{\varepsilon}n^2$ pairs $\{b,c\}$ with $b\in B, c\in C$. Thus,
combining this fact with Claim~\ref{cleaning2}, we have 
\begin{align}
\label{mindegreecheck1}
    &2\sum_{\{b,c\}\in E(L(x_1))}d(b,c) \leq 14\sqrt{\varepsilon}n^3 + 2 \sum_{\{b,c\}\in E(L_{B,C}(x_1))}d(b,c) \\ \leq&  16 \sqrt{\varepsilon}n^3 + 2 |L_{B,C}(x_1)| \left(\frac{n}{3} + 3\varepsilon^{1/4}n \right) \nonumber  \leq  7\varepsilon^{1/4}n^3 + \frac{4}{9}|B||C|n \leq  8\varepsilon^{1/4}n^3 + \frac{4}{81}n^3,
\end{align}
where we used Claim~\ref{classsizes1} in the last inequality. 
The number of vertices $y\in A$ satisfying $d(x_1,y)\geq \varepsilon^{1/4}n$ is at most $10\varepsilon^{1/4} n$ as otherwise one of the following three inequalities would hold
\begin{align*}
    |L_A(x_1)|>\sqrt{\varepsilon}n^2 \quad \text{or} \quad  |L_{A,B}(x_1)|, |L_{A,C}(x_1)|> 2\sqrt{\varepsilon}n^2,
\end{align*}
contradicting Claim~\ref{cleaning2} or Claim~\ref{cleaning5}. Further, the number of vertices $y\in B$ satisfying $d(x_1,y)\geq \frac{n}{3} + 3\varepsilon^{1/4}n$ is at most $4\varepsilon^{1/4} n$ as otherwise 
\begin{align*}
    |L_B(x_1)|>\sqrt{\varepsilon}n^2 \quad \text{or} \quad|L_{A,B}(x_1)| > 2\sqrt{\varepsilon}n^2,
\end{align*}
contradicting Claim~\ref{cleaning2} or Claim~\ref{cleaning5}. Similarly, we can conclude that the number of vertices $y\in C$ satisfying $d(x_1,y)\geq \frac{n}{3} + 3\varepsilon^{1/4}n$ is at most $4\varepsilon^{1/4} n$. Combining these three facts, we have 

\begin{align}
\label{mindegreecheck2}
    \sum_{y\in V(G)}d(x_1,y)^2 &\leq 18\varepsilon^{1/4} n^3 +|A| (\varepsilon^{1/4}n)^2+ (|B| + |C|) \left(\frac{n}{3} + 3\varepsilon^{1/4}n \right)^2 \nonumber \\
    &\leq 21\varepsilon^{1/4} n^3 + (|B| + |C|) \frac{n^2}{9} \leq \left(\frac{2}{27} + 23 \varepsilon^{1/4} \right)n^3,
\end{align}
where we used Claim~\ref{classsizes1} in the last inequality. Combining \eqref{mindegreecheck1} and \eqref{mindegreecheck2}, we can upper bound $q(x_1)$:
\begin{align*}
    q(x_1) &\leq  \sum_{y\in V(G)}d(x_1,y)^2 + 2\sum_{\{b,c\}\in E(L(x_1))}d(b,c) \leq \left(\frac{2}{27} + 23 \varepsilon^{1/4} \right)n^3 + \left( 8\varepsilon^{1/4} + \frac{4}{81} \right)n^3 \\
    &\leq \left(31 \varepsilon^{1/4} + \frac{10}{81} \right)n^3 < \frac{4}{27}n^3 - 10n^2 ,
\end{align*}
violating \eqref{minqassumption2}. Thus, there is no edge $x_1x_2v$ with $x_1,x_2\in A, v\in B$. The remaining part of this claim follows by a similar argument. 
\end{proof}

Now, by Claim~\ref{cleaning6} we have that $G$ only has edges with vertices in all three sets $A,B,C$. Thus,  
\begin{align}
\label{classsizes}
    \textup{co}_2(S_n) \leq \textup{co}_2(G)\leq |A||B||C|^2+|A||C||B|^2+ |B||C||A|^2=|A||B||C|n.
\end{align}
However, also $|A||B||C|n \leq \textup{co}_2(S_n)$.
Thus, equality must hold in \eqref{classsizes}. Equality only holds iff
\begin{align}
    |A|,|B|,|C|\in \left\{\left\lfloor \frac{n}{3} \right\rfloor,\left\lceil \frac{n}{3} \right\rceil \right\}
\end{align}
and all triples with one vertex from each set $A,B,C$ form edges in $G$.  
Thus $G\cong S_n$.

\end{proof}
We complete the proof of Theorem~\ref{F5exact} by inductively showing that the additional minimum degree type assumption \eqref{minqassumption2} is not more restrictive. The proof follows the idea of \cite[Theorem~1.7.]{BalCleLidmain}. We repeat the argument here for completeness.

\begin{proof}[Proof of Theorem~\ref{F5exact}]
Let $H$ be a $3$-uniform $F_{5}$-free hypergraph with codegree squared sum at least $\textup{co}_2(H)\geq \textup{co}_2(S_n)$. Set $d(n)=4/27n^3-10n^2$ and note that
\begin{align*}
    \textup{co}_2(S_n)-\textup{co}_2(S_{n-1})>d(n).
\end{align*} 
If all vertices $x\in V(G)$ satisfy \eqref{minqassumption2}, Theorem~\ref{F5mindegasumption} gives the result. Therefore, we can assume there exists a vertex $x\in V(G)$ not satisfying \eqref{minqassumption2}.
Remove $x$ with $q(x)<d(n)$ to get $G_{n-1}$ with
\begin{align*}
\textup{co}_2(G_{n-1})\geq \textup{co}_2(G_n)-q(x) \geq \textup{co}_2(G_n)-d(n) \geq \textup{co}_2(S_n)-d(n)\geq \textup{co}_2(S_{n-1})+1.
\end{align*}
Repeat this process as long as such a vertex exists. This gives us a sequence of hypergraphs $G_m$ on $m$ vertices with $\textup{co}_2(G_m)\geq \textup{co}_2(S_m)+n-m$. This process stops before we reach a hypergraph on $n_0=n^{1/4}$ vertices, because  $\textup{co}_2(G_{n_0})>n-n_0 > \binom{n_0}{2}(n_0-2)^2$ which is not possible.
Let $n'$ be the index of the hypergraph where this process stops. $G_{n'}$ satisfies $q(x)\geq d(n')$ for all $x\in V(G_{n'})$ and $\textup{co}_2(G_{n'})\geq \textup{co}_2(S_{n'})$ where the last inequality is strict if $n>n'$. Applying Theorem~\ref{F5mindegasumption} leads to a contradiction.

\end{proof}

\section{Loose Path and Cycle}
\label{sec:cycle}
Tur\'an problems for loose paths and loose cycles have been studied intensively. In the graph setting, the Erd\H{o}s-Gallai Theorem determines 
\begin{align}
\label{ErdosGallai}
    \textup{ex}(n,P_s^2)=(s-1)n
\end{align} when $s$ divides $n$. For uniformity $k\geq 4$, F\"{u}redi, Jiang and Seiver \cite{Furedipath} determined $\textup{ex}(n,P^k_s)$ exactly, for $n$ large enough. Kostochka, Mubayi and Verstra\"ete~\cite{KostochkCk} extended this result to the case $k=3$. \\
Cs\'ak\'any and Kahn~\cite{KahnC_3} determined $\textup{ex}(n,C_3^3)$; and for $s\geq3,k\geq 5$, F\"uredi and Jiang \cite{Furedicycle} determined the extremal function for $C_s^k$ and large enough $n$. Kostochka, Mubayi and Vestra\"ete~\cite{KostochkCk} extended this result for $k=3,4$. 
In this section we will determine the codegree squared ext\-re\-mal number of $P_s^3$ and $C_s^3$ asymptotically. We will make use of the asymptotic version of the previously mentioned results for these $3$-graphs. 
\begin{theo}[see e.g.~\cite{KostochkCk}]
\label{Kostochkapath}
Let $s\geq 4$. Then,
\begin{align*}
    \textup{ex}(n,C_s^3)=\left \lfloor \frac{s-1}{2}\right \rfloor \frac{n^2}{2}(1+o(1)) \quad \quad \text{and} \quad \quad \textup{ex}(n,P_s^3)=\left \lfloor \frac{s-1}{2}\right \rfloor \frac{n^2}{2}(1+o(1)).
\end{align*}
\end{theo}

The $n$-vertex $3$-graph containing all $3$-sets intersecting a fixed vertex set of size $\left\lfloor \frac{s-1}{2}\right\rfloor$ achieves the lower bound in Theorem~\ref{codcyclepath} and Theorem~\ref{Kostochkapath}.


In order to prove the upper bound in Theorem~\ref{codcyclepath} we will also use some of the lemmas Kostochka, Mubayi and Vestra\"ete needed to prove their result. Let $H$ be a $3$-uniform hypergraph on $n$ vertices. The hypergraph $H$ is $d$\textit{-full} if for every pair $x,y$ with $d_H(x,y)>0$ we have $d_H(x,y)>d$.
Denote by $\partial H$ the \emph{shadow graph} of $H$, i.e., the graph on $V(H)$ where a pair $\{x,y\}$ forms an edge if $d_H(x,y)>0$. 
\begin{lemma}[Lemma 3.1.~in \cite{KostochkCk}]
\label{sashadfull1}
For $d\geq 1$, every $n$-vertex $3$-uniform hypergraph $H$ has a $(d+1)$-full subhypergraph $F$ with 
$$ |F|\geq |H|-d|\partial H|.$$
\end{lemma}

\begin{lemma}[Lemma 3.2.~in \cite{KostochkCk}]
\label{sashadfull2}
Let $s\geq 3$ and let $H$ be a nonempty  $3s$-full $3$-uniform hyperpraph. Then $C_s^3,P_{s-1}^3\subseteq H$. 
\end{lemma}

\begin{lemma}[Lemma 5.1.~in \cite{KostochkCk}]
\label{sashadfull3}
Let $\varepsilon >0$ and $s\geq 4$. Let $H$ be a $3$-uniform hypergraph, and $F\subset E(\partial H)$ with $|F|>\varepsilon n^2$. Suppose that $d_H(f)\geq \lfloor \frac{s-1}{2} \rfloor+1$ for every $f\in F$ and, if $s$ is odd, then in addition, for every $f=xy\in F$ there is $e_f=xy\alpha \in H$ such that $\min\{d_H(x\alpha),d_H(y\alpha)\} \geq 2$ and $\max \{d_H(x\alpha),d_H(y\alpha)\}\geq 3s+1$. Then, for large enough $n$, $H$ contains $P_s^3$ and $C_s^3$.  
\end{lemma} 
Let $H$ be a $3$-uniform hypergraph. Recall that $w_H(e)=d_H(x,y)+d_H(y,z)+d_H(z,x)$ for $e=xyz\in E(H)$. We denote by $H_{2s}$ the subhypergraph of $H$ only consisting of edges $e$ satisfying $w_H(e)\geq 2n+2s$. Note that every edge $xy$ in the shadow graph $\partial H_{2s}$ has codegree $d_H(x,y)\geq 2s$ in $H$. 

\begin{lemma}
\label{badedges}
Let $s\geq4$. Let $H$ be a $3$-uniform hypergraph on $n$ vertices. If $H$ is $P_s^3$- or $C_s^3$-free, then $|E(H_{2s})|=o(n^2)$.   
\end{lemma}
\begin{proof}
We will first show that $|E(\partial H_{2s})|=o(n^2)$.

\medskip

\noindent
Case 1: $H$ is $P_s^3$-free.

The shadow graph $\partial H_{2s}$ does not contain a copy of $P_s^2$, because every edge of this path has codegree at least $2s$ in $H$ and thus one could find a $P^3_s$ in $H$. By the Erd\H{o}s-Gallai Theorem, see \eqref{ErdosGallai}, $|E(\partial H_{2s})|=O(n)=o(n^2).$

\medskip

\noindent
Case 2: $H$ is $C_s^3$-free and $s$ is even.

The shadow graph $\partial H_{2s}$ does not contain a copy of $C_s^2$, because every edge of this $s$-cycle has codegree at least $2s$ in $H$ and thus one could find a $C_s^3$ in $H$. Bondy and Simonovits' even cycle theorem~\cite{Bondyevencycle} states that an $n$-vertex graph with no copy of $C_s^2$ contains at most $O(n^{1+2/s})$ edges. Thus, $|\partial H_{2s}|=O(n^{1+2/s})=o(n^2)$. 
\medskip

\noindent
Case 3: $H$ is $C_s^3$-free and $s$ is odd.

Suppose $|\partial H_{2s}| > \varepsilon n^2$ for some $\varepsilon>0$. Apply Lemma~\ref{sashadfull3} on $H$ with $F=E(\partial H_{2s}) \subseteq E(\partial H)$. Note that the conditions of Lemma~\ref{sashadfull3} are satisfied: We have $d_H(f)\geq 2s\geq \lfloor \frac{s-1}{2}\rfloor +1$ for $f=xy\in F$. Since $f\in E(\partial H_{2s})$, there exists $e_f=xy\alpha\in E(H_{2s})$. So $w_H(e_f)\geq 2n+2s$ and in particular
\begin{align*}\min\{d_H(x\alpha),d_H(y\alpha)\} \geq w_H(e_f)-2n\geq 2s 
\end{align*}
and
\begin{align*}\max\{d_H(x\alpha),d_H(y\alpha)\} \geq \frac{d_H(x\alpha)+d_H(y\alpha)}{2} \geq \frac{w_H(e_f)-n}{2}\geq \frac{n+2s}{2}\geq 3s+1.
\end{align*}
We conclude that there is a copy of $C^3_s$ in $H$, a contradiction. Thus, $|\partial H_{2s}|=o(n^2).$

\medskip

\noindent
We obtained in all three cases $|\partial H_{2s}|=o(n^2)$. By Lemma~\ref{sashadfull1}, $H_{2s}$ has a $3(s+1)$-full subgraph $H'$  with 
$$ |E(H')|\geq |E(H_{2s})|-3s|E(\partial H_{2s})|.$$
By Lemma~\ref{sashadfull2}, $H'=\emptyset$ and thus $|E(H_{2s})|\leq 3s|E(\partial H_{2s})|=o(n^2)$.  

\end{proof}
\begin{proof}[Proof of Theorem~\ref{codcyclepath}]
 Let $H$ be a $P_s^3$- or $C_s^3$-free $3$-uniform hypergraph on $n$ vertices. By Lemma~\ref{badedges}, $|E(H_{2s})|=o(n^2)$. Therefore, we get
\begin{align*}\sum_{\{x,y\}\in \binom{[n]}{2}}d_H^2(x,y)&= \sum_{e \in E(H_{2s})}w_H(e)+\sum_{e  \in E(H) \setminus E(H_{2s}) }w_H(e) \leq o(n^3)+ (2n+2s)|E(H)|\\
&\leq \left\lceil \frac{s-1}{2}\right\rceil n^3(1+o(1)), \end{align*}
where in the last inequality we used Theorem~\ref{Kostochkapath}
\end{proof}

We raise the problem of determining for $P_s^3$ and $C_s^3$ exactly.
\begin{problem}
Let $s\geq 3$. Determine $\textup{exco}_2(n,P_s^3)$ and $\textup{exco}_2(n,C_s^3)$ as a function of $n$ and $s$.
\end{problem}
For $C_3^3$, we solve this problem.
\begin{theo}
Let $n\geq 6$. Then,
\begin{align*}
\textup{exco}_2(n,C_3^3)= \binom{n-1}{2}(2n-3).
\end{align*}
\end{theo}
\begin{proof}
Let $H$ be a $C_3^3$-free $3$-graph on $n$ vertices.
Cs\'ak\'any and Kahn~\cite{KahnC_3} proved that $|E(H)|\leq \binom{n-1}{2}$.
Assume that there is an edge $e=xyz$ with $w(e)>2n-3$, so one can find distinct vertices $a,b,c$ with $a\in N(x,y), \ b\in N(x,z)$ and $c\in N(y,z)$. Now, $\{xay,ycz,zbx\}$ forms a $C_3^3$. Thus, we can assume that $w(e)\leq 2n-3$ for all edges. This allows us to conclude
\begin{align*}
\textup{co}_2(H)= \sum_{e\in E(H)}w(e) \leq |E(H)|(2n-3)\leq \binom{n-1}{2}(2n-3).
\end{align*}
This proves the upper bound. Now, for the lower bound, consider the 3-graph $H$ containing all $3$-sets intersecting in one fixed vertex. This $3$-graph is $C_3$-free and has codegree squared sum 
\begin{align*}
\textup{co}_2(H)= \binom{n-1}{2} 1^2 + (n-1)(n-2)^2 = \binom{n-1}{2}(2n-3).
\end{align*}
\end{proof}

\section{Matchings}
\label{sec:matching}
In this section we prove Theorem~\ref{codmatching}. Denote by $M_s^k$ the $k$-uniform matching of size $s$, i.e., the $k$-uniform hypergraph on $ks$ vertices with $s$ disjoint edges. For graphs, Erd\H{o}s and Gallai \cite{Erdosmatching1} proved that the extremal number for $M_s^2$ is 
\begin{align}
     \textup{ex}(n,M_s^2)=\max \left \{\binom{2s-1}{2}, (s-1)(n-s+1)+\binom{s-1}{2}\right \}. \label{graphmatching} 
\end{align}
Erd\H{o}s \cite{Erdosmatching2} also determined the extremal number in the hypergraph case when the number of vertices is sufficiently large.  

\begin{theo}[Erd\H{o}s \cite{Erdosmatching2}]
Let $k\geq 2,s\geq 1$ be integers. Then there exists a constant $c=c(s)$ such that for all $n\geq ck$
\begin{align*}
\textup{ex}(n,M_s^{k})= |E(G(n;k,s-1)|=\sum_{i=1}^{\min\{k,s-1\}} \binom{s-1}{i}\binom{n-s+1}{k-i},
\end{align*}
where $G(n;k,s-1)$ is the $k$-uniform hypergraph with vertices $x_1,\ldots,x_n$ and all edges containing at least one of the vertices $x_1,\ldots,x_{s-1}$.
\end{theo}

We will use the extremal result for matchings in graphs to determine the codegree squared extremal number for the matching asymptotically. 

\begin{proof}[Proof of Theorem~\ref{codmatching}]
The lower bound is achieved by $G(n;3,s-1)$.
\begin{align*}
     \textup{co}_2(G(n;3,s-1)) &= \binom{s-1}{2} (n-2)^2 + (s-1)(n-s+1)(n-2)^2 + \binom{n-s+1}{2} (s-1)^2 \\ &= (s-1)n^3 (1+o(1)).
\end{align*}
For the upper bound, let $H$ be an $M_s^3$-free $3$-uniform hypergraph. We construct an auxiliary graph $G$ with the same vertex set $V(H)$. A pair $xy$ is an edge in $G$ iff  $d_H(x,y)\geq 3s$. $G$ is $M_s^2$-free, because if there is a matching of size $s$ in $G$, then it can be extended to a $3$-uniform matching of size $s$ in $H$. By \eqref{graphmatching} we have $|E(G)|\leq n(s-1) (1+o(1))$.
Now, we can upper bound the codegree squared sum:  
\begin{align*}
   \textup{co}_2(H)&= \sum_{xy\in E(G)}d^2_H(x,y) + \sum_{xy\notin E(G)}d^2_H(x,y)  \leq |E(G)| n^2 + \binom{n}{2}(3s)^2= (s-1)n^3(1+o(1)). 
\end{align*}
Thus,
\begin{align*}
    \textup{exco}_2(n,M_s^3)=(s-1)n^3 (1+o(1)).
\end{align*} 
\end{proof}

\section{Star}
\label{sec:star}
Chung and Frankl \cite{Franklstar} determined the Tur\'an number of a star.

\begin{theo}[Chung, Frankl \cite{Franklstar}]
\label{extremalstar}
Let $s\geq 3$. If $s$ is odd and $n\geq s(s-1)(5s+2)/2$, then 
\begin{align*}
    \textup{ex}(n,S_s^3)=(n-2s)s(s-1)+ 2\binom{s}{3}.
\end{align*} 
If $s$ is even and $n\geq 2s^3-9s+7$, then
\begin{align*} 
\textup{ex}(n,S_s^3)=\left(s^2-\frac{3}{2}s \right)n-\frac{1}{2}(2s^3-9s+6).
\end{align*}
\end{theo}
They also determined the extremal example, which depends on the parity of $s$. 
\begin{exmp}[Chung, Frankl \cite{Franklstar}]
Their extremal example in the case when $s$ is odd is the following hypergraph $F_o$. The vertex set is the disjoint union of the sets $A,B,C$ with sizes $s,s$ and $n-2s$ respectively. The edge set is 
$$ E(F_o)=\{e: |e\cap A| \geq 2, e\cap B=\emptyset  \} \cup \{e: |e\cap B| \geq 2, e\cap A=\emptyset  \}. $$
\end{exmp}
\begin{exmp}[Chung, Frankl \cite{Franklstar}]
Their extremal example in the even case is constructed as follows. First, consider the auxiliary graph $G$ on $2s-1$ vertices $x_1,\ldots ,x_{s-1},y_1,\ldots ,y_{s-1},z$ with edges being all pairs $\{x_i,y_j\}$ with $i\neq j$, $\{x_i,y_i\}$ for $2i\leq s$ and $\{x_i,z\}, \{y_i,z\}$ for $2i>s$. Now construct $F$ on $n$ vertices with edges being all $3$-sets intersecting $V(G)$ in one edge or containing two edges of $G$. The extremal example in the even case is the hypergraph $F_e$ constructed by adding the edges of the form $x_1y_iz$ to $F$.   
\end{exmp}

We will determine the codegree squared extremal number of the star asymptotically by using their result.

\begin{proof}[Proof of Theorem~\ref{codstar}]
The lower bound is achieved by the two extremal examples presented above. We have 
\begin{align*}
    \textup{co}_2(F_o)= 2 \left (\binom{s}{2}(n-s-2)^2+s(n-2s)(s-1)^2 \right)=s(s-1)n^2 (1+o(1))
\end{align*} 
and
\begin{align*} \textup{co}_2(F_e)=\left( s^2-\frac{3}{2}s\right)n^2(1+o(1)).
\end{align*}
For the proof of the upper bound let $H$ be an $n$-vertex $S_s^3$-free $3$-graph. Recall that the weight of an edge $e=xyz\in E(H)$ is defined to be
    \[
    w_H(e)=d(x,y)+d(x,z)+d(y,z).
    \]
Consider the subgraph $H'\subset H$ only consisting of the edges $e\in E(H)$ with $w_H(e)\geq n+4s$. Since $H'$ is a subgraph of $H$, $H'$ is clearly $S_s^3$-free. Further, we claim that $H'$ is $M_{3s}^3$-free and satisfies $d_{H'}(x,y)<2s$ for all $x,y\in V(H)$.
\begin{claim}
$H'$ is $M_{3s}^3$-free.
\end{claim}
\begin{proof}
Assume $H'$ contains a copy $M$ of $M_{3s}^3$. Then
\begin{align*}
    \sum_{\substack{\{x,y\}\subset e \\ \text{for some } e\in M }} d(x,y)\geq (n+4s)3s,
\end{align*}
and therefore there exists a vertex $v\in V(H)$ which is contained in at least $3s$ edges of type $\{v,x,y\}$ with $\{x,y\}\subset f$ for some $f \in M$. Thus, $v$ is contained in at least $s$ edges only having $v$ as a pairwise intersection and therefore is the center of a copy of $S_s^3$, a contradiction. We conclude that $H'$ is $M_{3s}^3$-free.  
\end{proof}
\begin{claim}
\label{codegreesmall}
$H'$ satisfies $d_{H'}(x,y)<2s$ for all $x,y\in V(H)$.
\end{claim}
\begin{proof}
Assume that there exists $x,y\in V(H')$ such that $d_{H'}(x,y)\geq 2s$. Let $v_1,\ldots,v_{2s}$ be vertices such that $v_ixy\in E(H')$. For each $i$, $d_H(x,v_i)>2s$ or $d_H(y,v_i)>2s$ since $w_H(e)>n+4s$ for all $e\in E(H')$. Thus the number of indices $i\in [2s]$ such that $d_H(x,v_i)>2s$ or the number of indices $i\in [2s]$ such that $d_H(y,v_i)>2s$ is at least $s$. This contradicts that $H$ is $S_{s}^3$-free (see e.g. \cite [Lemma~2.1]{Franklstar}). We conclude that $d_{H'}(x,y)<2s$ for all $x,y\in V(H)$. 
\end{proof}
Next, we bound the number of edges of $H'$. 
\begin{claim}
\begin{align*}
    |E(H')|\leq (12s^2+6s+1)3s.
\end{align*}
\end{claim}
\begin{proof}
Let $e\in E(H')$ be an arbitrary edge. By Claim~\ref{codegreesmall}, the number of edges in $H'$ intersecting $e$ in exactly two elements is at most $6s$. Furthermore, the number of edges in $H'$ intersecting $e$ in exactly one element is at most $12s^2$. To verify this, let $v\in e$ and consider the link graph $L(v)$ in $H'$. By Claim~\ref{codegreesmall}, $L(v)$ has maximum degree at most $2s$. Further, $L(v)$ cannot have a matching of size $s$, as otherwise it forms a star in $H'$ with $v$ being the center. A graph with maximum degree at most $2s$ and no matching of size $s$ has at most $4s^2$ edges. Thus, $e$ is incident to at most $4s^2$ edges in $H'$ intersecting $e$ in $v$. This allows us to conclude that $e$ is incident to at most $12s^2+6s$ edges in $H'$. Since $H'$ is $M_{3s}^3$-free, it can have at most $(12s^2+6s+1)3s$ edges. 
\end{proof}

Using the bounds on the edges in $H'$ and $H$, we can get an upper bound for the codegree squared sum of $H$: \begin{align*}
    \textup{co}_2(H) &= \sum_{e\in E(H)} w_H(e) = \sum_{e\in E(H')} w_H(e) + \sum_{e\in E(H) \setminus E(H')} w_H(e) \\
   &\leq |E(H')| 3n + |E(H)| (n+4s) \leq (12s^2+6s+1)3s (3n)+ |E(H)| (n+4s).   
\end{align*}
Thus, by using Theorem~\ref{extremalstar}, if $s$ is odd, then
\begin{align*}
 \textup{co}_2(H) \leq s(s-1)n^2(1+o(1))
\end{align*}
and if $s$ is even, then
\begin{align*}
 \textup{co}_2(H) \leq \left(s^2-\frac{3}{2}s \right) n^2(1+o(1)).
\end{align*}
\end{proof}

\section{Further Discussion}
\label{openques}

\subsection{Positive minimum codegree}
There are more notions of extremality which could be studied. An interesting variant of the minimum codegree threshold is the positive minimum codegree which very recently was introduced by Halfpap, Lemons and Palmer~\cite{personalComCory}. Given a $3$-graph $G$, the positive codegree $\textup{co}^+(G)$ is the minimum of $d(x,y)$ over all pairs ${x,y}$ with positive codegree, i.e. $d(x,y)>0$. Given a $3$-graph $H$, the \emph{positive minimum codegree} of $H$, denoted by $\textup{co}^+\textup{ex}(n,H)$, is the largest minimum positive co-degree in an $F$-free $n$-vertex graph. 

\begin{ques}[Halfpap, Lemons and Palmer~\cite{personalComCory}]
Given a $3$-graph $H$, what is the \emph{positive minimum codegree} of $H$. 
\end{ques}
The positive minimum codegree of $K_4^{3-}$ and $F_5$ is obtained by $S_n$, i.e. 
\begin{align*}
    \textup{co}^+\textup{ex}(n,K_4^{3-})=\textup{co}^+\textup{ex}(n,F_5)=\left\lfloor \frac{n}{3} \right\rfloor.
\end{align*}
This shows that the positive minimum codegree behaves significantly different than the minimum codegree threshold. An interesting property Halfpap, Lemons and Palmer~\cite{personalComCory} proved is that $\textup{co}^+\textup{ex}(n,H) = o(n)$ if and only if $H$ is $3$-partite. This implies that if $H$ satisfies  $\textup{co}^+\textup{ex}(n,H)=\Theta(n)$, then in fact $\textup{co}^+\textup{ex}(n,H)\geq \frac{n}{3} (1+o(1))$. This means that the scaled positive minimum codegree jumps from $0$ to $1/3$.

\subsection{Higher Uniformities}
We propose to study $\sigma(H)$ for $k$-uniform hypergraphs with $k\geq 4$. In particular it would be interesting to determine the codegree squared density for $K_5^4$, the $4$-uniform hypergraph on five vertices with five edges. For the Tur\'an density, Giraud~\cite{MR1077144} proved $\pi(K_5^4)\geq 11/16$ using the following construction. Let $A$ be an $n/2 \times n/2$ $\{0,1\}$-matrix. Define the $4$-graph $H(A)$ on $n$ vertices corresponding to the rows and columns of $A$. 
The edge set consists of two types of edges: any four vertices with exactly three rows or three columns; and any four vertices of two rows and two columns forming a $2 \times 2$ submatrix with an odd sum.
The best-known upper bound is $\pi(K_5^4)\leq \frac{1753}{2380}$ which is due to Markstr\"{o}m~\cite{MR2548922}.

For the codegree squared density, we get
\begin{align*}
    0.484375=\frac{31}{64}\leq \sigma(K_5^4)\leq 0.55241250, 
\end{align*}
where the lower bound comes from the same construction and the upper bound
from flag al\-geb\-ras. We conjecture the lower bound to be tight.
\begin{conj}
    We have
    \begin{align*}
        \sigma(K_5^4)=\frac{31}{64}.
    \end{align*}
\end{conj}

Another $4$-graph where we can determine the codegree squared density is $K_5^{4=}$, the unique $4$-uniform hypergraph on five vertices with exactly three edges.
\begin{theo}
\label{5vertices3edges}
Let $n$ be an integer divisible by $4$. Then,
\label{codK54=}
\begin{align*}
\textup{exco}_2(n,K_5^{4=})=\frac{n^2}{16} \binom{n}{3}.
\end{align*}
\end{theo}
Gunderson and Semeraro~\cite{K54doppleminusgunderson} presented a construction of a $4$-uniform hypergraph on $n$ vertices, where $n$ is divisible by $4$: every $3$-subset is contained in exactly $n/4$ hyperedges and every $5$ vertices span $0$ or $2$ edges. Denote by $G$ this $4$-uniform hypergraph. Then, $e(G)= n \binom{n}{3}/16$ and $\textup{co}_2(G)= n^2 \binom{n}{3}/16$. Furthermore, they proved that 
\begin{align}
\label{Gunderson}
    \textup{ex}(n,K_5^{4=})= \frac{n}{16} \binom{n}{3}
\end{align}
for $n$ divisible by $4$ by modifying a double counting idea by de Caen~\cite{MR734038}. Note that Gunderson and Semeraros' result~\cite{K54doppleminusgunderson} also implies the codegree threshold result $\textup{ex}_3(n,K_5^{4=})=\frac{n}{4}$ for $n$ divisible by $4$. Here, we will modify de Caen's argument~\cite{MR734038} to prove Theorem~\ref{5vertices3edges}. 
\begin{proof}[Proof of Theorem~\ref{5vertices3edges}]
Let $H$ be a $4$-uniform $n$-vertex $K_5^{4=}$-free hypergraph. Denote $N$ the number of pairs $(e,f)$ with $|e|=|f|=4$, $|e\cap f|=3$, $e\in E(H)$ and $f\notin E(H)$. Then, on one side
\begin{align}
\label{doupecountK54=1}
    N=\sum_{e\in E(H)} \sum_{x\not\in e} |\{z\in e: e \cup \{ x \} \setminus  \{z\} \not\in E(H)\}| \geq \sum_{e\in E(H)} \sum_{x\not\in e} 3= 3e(H)(n-4),
\end{align}
where the inequality holds because $H$ is $K_5^{4=}$-free.
On the other side,
\begin{align}
\label{doupecountK54=2}
    N= \sum_{A\in \binom{[n]}{3}}d_H(A)(n-3-d_H(A)) = (n-3)\sum_{A \in \binom{[n]}{3}} d_H(A) - \textup{co}_2(H)= 4(n-3)  e(H)-  \textup{co}_2(H).
\end{align}
Combining \eqref{doupecountK54=1} and \eqref{doupecountK54=2}, we get 
\begin{align*}
    \textup{co}_2(H)\leq 4(n-3)e(H)- 3e(H)(n-4) = 
     ne(H) \leq \frac{n^2}{16} \binom{n}{3},
\end{align*}
where in the last inequality we used \eqref{Gunderson}.
\end{proof}
We remark that $K_5^{4=}$ is the only non-trivial hypergraph we are aware of which has the property that there is a construction (presented in \cite{K54doppleminusgunderson}) that is extremal in each of $\ell_1$-norm, $\ell_2$-norm and for the codegree threshold.

\subsection{Big Triangle}
Let $T^{(2k)}$ be the $2k$-graph obtained by having three disjoint sets $A,B,C$ of size $k$ and the three edges $A\cup B,\ A \cup C,\ C \cup A$. Frankl~\cite{fattrianglefrankl} proved that $\pi(T^{(2k)})=1/2$. The lower bound comes from the $2k$-graph $G$ where the vertex set is partitioned into two equal sized parts and edges are those $2k$-sets which intersect each part in an odd number of elements. The minimum codegree of this extremal example is $n/2-o(n)$. Mubayi and Zhao~\cite{Mubayicode} observed that in fact $\pi_{2k-1}(T^{(2k)})=1/2$. 

The maximum sum of $k$-codegrees squared an $n$-vertex $2k$-uniform $T^{(2k)}$-free graph can have, can be easily calculated. Let $H$ be a $T^{(2k)}$-free $2k$-graph. Then,

\begin{align*}
\sum_{\substack{A\subset V(H) \\ |A|=k}}d(A)^2 &=  \sum_{B\in E(H)}\sum_{\substack{A\subset B \\ |A|=k}} d(A)\leq |E(H)| \binom{2k}{k} \frac{1}{2}\binom{n}{k} (1+o(1)) \\
&\leq \frac{1}{2}\binom{n}{2k} \binom{2k}{k} \frac{1}{2}\binom{n}{k}(1+o(1))=\frac{1}{4}\binom{n}{k}^3 (1+o(1)).
\end{align*}
On the other side the $2k$-graph $G$ satisfies
\begin{align*}
\sum_{\substack{A\subset V(G) \\ |A|=k}}d_{G}(A)^2 \geq\binom{n}{k} \left( \frac{\binom{n-k}{k-1}(\frac{n}{2}-2k) }{k}  \right)^2 =\frac{1}{4}\binom{n}{k}^3 (1+o(1)).
\end{align*}
This determines asymptotically the maximum $k$-codegree squared sum of an $n$-vertex $2k$-uniform $T^{(2k)}$-free hypergraph. We wonder whether the asymptotic extremal example for the codegree squared sum is also $G$.
\begin{ques}
What is $\sigma(T^{(2k)})$? 
\end{ques}

\subsection{Induced Problems}
We propose to study induced problems. 
\begin{ques}
Given a family of $k$-graphs $\mathcal{F}$, what is the maximum $\ell_2$-norm a $k$-uniform $n$-vertex hypergraph $G$ can have without containing any $H\in \mathcal{F}$ as an induced subhypergraph?
\end{ques}
Let $\mathcal{F}$ be a family of $k$-uniform hypergraphs. Denote by $\textup{exco}_2^{ind}(n,\mathcal{F})$ the maximum codegree squared sum among all $k$-uniform $n$-vertex hypergraphs not containing any $F\in \mathcal{F}$ as an induced subhypergraph. Here, we will present two examples where we can determine $\textup{exco}_2^{ind}(n,\mathcal{F})$. 

Denote by $K_5^{4-}$ the $4$-uniform hypergraph with exactly $4$ edges on $5$ vertices and $E_5^{4}$ the $4$-uniform hypergraph with exactly one edge on five vertices. Observe that the proof of Theorem~\ref{5vertices3edges} actually gives the stronger result 
\begin{align*}
    \textup{exco}_2^{ind}(n,\{E_5^{4},K_5^{4=}, K_5^{4-}, K_5^4\})=\frac{n^2}{16} \binom{n}{3}
\end{align*}
for $n$ divisible by $4$.

Frankl and F\"uredi~\cite{K43-extremalfrankl} proved that the $3$-graph with the maximum number of edges only containing exactly $0$ or $2$ edges on any $4$ vertices is the $n$-vertex blow-up of $S_6$. As observed in Subsection~\ref{K43-F32}, this $3$-graph has codegree squared sum of $5/108n^4 (1+o(1))$. Using flag algebras for the upper bound, we get 
\begin{align*}
     \textup{exco}_2^{ind}(n,\{E_4^{3},K_4^{3-},K_4^3\})=\frac{5}{108}n^4 (1+o(1)),
\end{align*}
where $E_4^{3}$ denotes the $3$-graph with exactly one edge on four vertices.   

\section*{Acknowledgments}
We thank Emily Heath and Haoran Luo for useful comments and suggestions. 
\bibliographystyle{abbrvurl}
\bibliography{codegreesquare2}

%

\end{document}